\PassOptionsToPackage{square,sort&compress}{natbib}
\documentclass[10pt,number,preprint,times]{elsarticle}
\usepackage{physics}
\makeatletter
\def\ps@pprintTitle{%
	\let\@oddhead\@empty
	\let\@evenhead\@empty
	\let\@oddfoot\@empty
	\let\@evenfoot\@oddfoot
}
\makeatother
\makeatletter

\usepackage{amssymb}
\usepackage{atbegshi}
\AtBeginDocument{\AtBeginShipoutNext{\AtBeginShipoutDiscard}}
\usepackage[utf8]{inputenc} 
\usepackage{alltt}
\usepackage{lipsum}
\usepackage[numbers,sort&compress,square]{natbib}
\usepackage{amsfonts}
\usepackage{mathtools}
\usepackage{bbm, dsfont}

\usepackage{amsmath}
\usepackage{times}
\usepackage{tikz-cd}
\usepackage{manyfoot}%

\usepackage{booktabs}
\usepackage[T1]{fontenc}
\usepackage[english,french]{babel}
\usepackage{lmodern}
\usepackage[a4paper]{geometry}
\usepackage{epstopdf}
\usepackage{mathrsfs}
\usepackage{latexsym,amssymb,amsfonts}
\usepackage{enumerate}
\usepackage{relsize,exscale}
\usepackage{dsfont}
\usepackage{microtype}
\usepackage{xcolor}
\usepackage{enumitem} 
\usepackage{graphicx}
\usepackage{float}
\usepackage{upgreek}
\usepackage{caption}
\usepackage{subfigure}

\usepackage[labelfont=bf]{caption}
\usepackage{xcolor}
\definecolor{mred0}{rgb}{0,0.406,0.596}
\definecolor{mcyan}{rgb}{0,0.501,0.501}
\definecolor{mcyan}{rgb}{0,0.501,0.501}
\definecolor{mblack}{rgb}{0,0.406,0.796}
\definecolor{mcgreen}{rgb}{0.180,0.545,0.41}
\definecolor{mred0}{rgb}{0,0.406,0.596}
\definecolor{mred3}{rgb}{0.6,0,0.8}
\definecolor{DarkGreen}{rgb}{0.278,0.701,0.913} 
\definecolor{mred1}{rgb}{0.180,0.545,0.341}

\usepackage{longtable}

\numberwithin{equation}{section} 

\usepackage{amsthm}
\newtheorem{thm}{Theorem}[section]
\newtheorem{defi}{Definition}[section]
\newtheorem{cor}{Corollary}[section]
\newtheorem{lem}{Lemma}[section]
\newtheorem{rem}{Remark}[section]

\usepackage{float}
\usepackage{multicol}
\usepackage[colorlinks=true,linkcolor=black,citecolor=green,urlcolor=red]{hyperref}
\usepackage[figurename=Fig.]{caption}  

\newcommand{\E}{\ensuremath{\Bbb{E}}}

\date{}
\begin{document}
	\maketitle
	\selectlanguage{english}
		\begin{abstract}
		The purpose of this paper is to establish the well-posedness of martingale (probabilistic weak) solutions to stochastic degenerate aggregation–diffusion equations arising in biological and public health contexts. The studied equation is of a stochastic degenerate parabolic type, featuring a nonlinear two-sidedly degenerate diffusion term accounting for repulsion, a locally Lipschitz reaction term representing competitive interactions, and a stochastic perturbation term capturing environmental noise and uncertainty in biological systems. The existence of martingale solutions is proved via an auxiliary nondegenerate stochastic system combined with the Faedo–Galerkin method. Convergence of approximate solutions is established through Prokhorov’s compactness and Skorokhod’s representation theorems, and uniqueness is obtained using a duality approach. Finally, numerical simulations are given to illustrate the impact of environmental noise on aggregation dynamics and the long-term behavior of the system, offering insights that may inspire medical innovation and predictive modeling in public health. \\[1ex] 
	\end{abstract}
	\begin{frontmatter}
\title{Global weak martingale solutions to a stochastic two-sidedly degenerate aggregation-diffusion equation issued from biology}

	\author{Mostafa Bendahmane \corref{mycorrespondingauthor}}
	\address{Institut de Mathématiques de Bordeaux, Université de Bordeaux, 33076 Bordeaux Cedex, France}
	\ead{mostafa.bendahmane@u-bordeaux.fr}
	\cortext[mycorrespondingauthor]{Corresponding author:  mostafa.bendahmane@u-bordeaux.fr}

	\author{Mohamed Mehdaoui}
	\address{Euromed University of Fez, UEMF, Fez, Morocco}
	\address{MAMCS Group, FST Errachidia, Moulay Ismail University of Meknes, P.O. Box 509, Boutalamine 52000, Errachidia, Morocco}
		\ead{m.mehdaoui@edu.umi.ac.ma, m.mehdaoui@ueuromed.org}

	\author{Mouhcine Tilioua}
	\address{MAMCS Group, FST Errachidia, Moulay Ismail University of Meknes, P.O. Box 509, Boutalamine 52000, Errachidia, Morocco}

	\ead{m.tilioua@umi.ac.ma}
\end{frontmatter}
	\footnotetext[1]{A part of this work was conducted during M. Mehdaoui's three-month visit to \textit{Institut de Mathématiques de Bordeaux (IMB), Université de Bordeaux, France}. His visit was funded by the French Embassy in Morocco and the Moroccan Center for Scientific and Technical Research (CNRST).}
	\newpage
	\section{Introduction and motivation}\label{intro}
Aggregation refers to the process in which living organisms cluster or come together to form a larger structure or mass. Understanding the key features of such a process has been a mainstream topic for theoretical biologists due to its manifestations across multiple sub-disciplines, including cell biology, microbiology, ecology, immunology, and behavioral biology. To mathematically capture these phenomena, remarkable developments in continuum physics have led to the formulation of aggregation–diffusion equations, which describe the process of aggregation while accounting for organism mobility. Beyond their theoretical interest, such models have important implications for public health, guiding strategies in disease prevention, health promotion, and healthcare access, as well as informing medical innovation. We refer the reader, for example, to \cite{di2024kermack,mogilner1999non,bodnar2006integro,burger2008large,laurent2007local,bertozzi2009blow,carrillo2019aggregation,gomez2024beginner} and the references therein. Given a bounded open set $\mathcal{O} \subset \mathbb{R}^d; (d\in {2,3})$ with smooth boundary $\partial \mathcal{O}$ and a time horizon of length $T>0$, such equations are typically of the following form:
\begin{equation}\label{basic}
	\begin{cases}
		\begin{split}
	&\partial_t u-\div\left(a(u) \nabla u-u \nabla \mathcal{K} \ast u \right)=f(u), &&\quad \text{in } Q_T:=\mathcal{O} \times (0,T),\\
	&u(.,0)=u_0,  &&\quad \text{in } \mathcal{O},\\
	&\left(a(u)\nabla u-u \nabla \mathcal{K} \ast u\right) \cdot \overrightarrow{\textbf{n}}=0, \;\; &&\quad \text{on } \Sigma_T:=\partial \mathcal{O} \times (0,T), 
	    \end{split}
    \end{cases}
\end{equation}
wherein $\overrightarrow{\textbf{n}}$ is the outward normal vector and $u$ denotes the density of living organisms throughout $Q_T$ starting from an initial state $u_0$. Moreover, $\partial_t$ and $\div$ stand for the partial derivative with respect to time $t \in (0,T)$ and the divergence operator with respect to $x \in \mathcal{O}$, respectively. The main features of Model \eqref{basic} are threefold. The first is manifested by $\mathcal{K} : \mathbb{R}^d \longrightarrow \mathbb{R}$ which is a kernel describing the attraction law. The latter contributes to the flux through the convolution term $\nabla \mathcal{K} \ast u (x):= \displaystyle \int_\mathcal{O} \nabla \mathcal{K}(x-y)u(y)dy,\; \forall x \in \mathcal{O}.$ The second feature resides in the Fickian diffusion rate $a : \mathbb{R} \longrightarrow \mathbb{R}^+$, which in simplest cases can be assigned constant values $(a(\cdot)\equiv a \geq 0)$. Finally, the last feature is exhibited by the function $ f :  \mathbb{R} \longrightarrow \mathbb{R}$, which captures the growth of living organisms as well as their competition for limited resources such as food, space, or mates.
The first emerging question when dealing with aggregation-diffusion models falling into Model \eqref{basic} pertains to the existence and, if possible, the uniqueness of biologically-feasible solutions. In this regard, notable contributions have been made in the case where $a(\cdot)=a \equiv 0$. However, the existing literature still exhibits critical gaps, particularly concerning nonlinear degenerate diffusion and stochastic influences, which we address next.

\subsection{The general case of a nonlinear density-dependent diffusion coefficient $a : \mathbb{R} \longrightarrow \mathbb{R}^+$}
This case is of particular interest and remains less explored in aggregation-diffusion equations. As a matter of fact, it has only been studied in other biological contexts in the case where the model is deprived of aggregation $(\mathcal{K} \equiv 0)$. In these contexts, the spatial dispersal of organisms is motivated by external factors such as the avoidance of crowding, migration into sparse areas, etc. which cannot be captured merely by the standard random diffusion $a(\cdot)\equiv a>0$ \cite{gurtin1977diffusion,newman1980some,bendahmane2007two,mehdaoui2023analysis,mehdaoui2024optimal,mehdaoui2024new}. In the presence of aggregation $(\mathcal{K} \neq 0)$, the case of a  density-dependent diffusion coefficient has been recently considered by Bendahmane et al. \cite{bendahmane2023optimal}. More precisely, the authors established the existence and uniqueness of positive essentially-bounded weak solutions to Model \eqref{basic} with the following underlying setting:
\begin{enumerate}[label={({A${{_\arabic*}}$})}]
	\item \label{a1} $a \in C^1(\mathbb{R}), a(0)=a(\bar{u})=0$ and $\exists \bar{u}>0,\; a(s)>0,\; \forall 0<s<\bar{u}$;
	\item \label{a2}  $K \in C^2\left(\mathbb{R}^d\right)$ is  non-negative and radially non-increasing such that $$\|K\|_{C^2\left(\mathbb{R}^d\right)}<\infty \text{ and } \displaystyle \int_{\mathbb{R}^d} K(x) dx=1;$$
	\item \label{a3} $f(u):=\alpha u-\mu u^2, \forall u \in \mathbb{R}^+$, where $\alpha>0$ and $\mu>0$ stand for the Malthusian growth rate and the competition rate, respectively. 
\end{enumerate}
From the mathematical modeling perspective, let us briefly underline that Assumption \ref{a1} captures two biological facts. On the one hand, that for a null density of living organisms, the spatial diffusion vanishes and, on the other hand, it also incorporates the so-called \textit{volume-filling effect}, which was previously introduced in \cite{laurenccot2005chemotaxis,wrzosek2006long,ainseba2008reaction,wang2012global}. Such an effect pertains to scenarios where the environment becomes crowded or densely-packed causing the diffusion of living organisms to vanish because there is less free volume or space available for movement.
\subsection{The case of aggregation-diffusion equations in stochastic environments}
It is common knowledge that random fluctuations play an important role in the dynamics of various biological systems \cite{mao2002environmental,gray2011stochastic,collins2012optimal,dhariwal2019global,mehdaoui2023dynamical,mehdaoui2023analysisstoch,mehdaoui2024well,bendahmane2023stochastic,sabbar2024exploring,sabbar2024probabilistic}. Aggregation-diffusion equations are not exempt from this fact. Indeed, when proceeding to calibrate the parameters of such equations from noisy measurements or observations, one needs to know in advance how would the model react to noise and if various established properties in the deterministic case (in the absence of noise) would still hold. This gives rise to the question of extending previous results in regards to mathematical well-posedness, stability, long-time behavior to the stochastic case with different types of noise. To the best of the authors' knowledge and as far as the extension of aggregation-diffusion models falling into Model \eqref{basic} to the stochastic case is concerned, the only published work on the mathematical well-posedness is the very recent one done by Tang and Wang \cite{tang2024strong} where the case of unbounded spatial domains was explored in the absence of a reaction term $(f \equiv 0)$ and for a constant diffusion coefficient $(a \equiv 1)$. More precisely, the authors addressed questions in relation to both local and global existence and uniqueness of strong solutions in both cases of a nonlinear multiplicative noise as well as a non-autonomous linear noise. 

Building on this, our central contribution of this paper is to establish new well-posedness results for a class of stochastic degenerate aggregation-diffusion models where the aforementioned scarcities are surpassed.
Note that our work extends the model developed in \cite{bendahmane2023optimal} to the stochastic case driven by cylindrical Wiener processes, and, in the case of bounded domains, generalizes the model considered in \cite{tang2024strong} to include nonlinear degenerate diffusion together reaction terms. Let us also mention that the extension of the model in \cite{tang2024strong} to the case of a nonlinear non-degenerate diffusion coefficient for bounded spatial domains has not been done. However, since the degenerate case is more challenging, we directly consider it here knowing that the non-degenerate case can be treated by simpler arguments. That being said, we consider the following model:
\begin{equation}\label{basicstoch}
	\begin{cases}
		\begin{split}
			&du=\left(\div\left(a(u) \nabla u-u \nabla \mathcal{K} \ast u\right) +f(u)\right)dt+\sigma(u(t))d\mathcal{W}_t, &&\quad \text{in } Q_T:=\mathcal{O} \times (0,T),\\
			&u(.,0)=u_0,  &&\quad \text{in } \mathcal{O},\\
			&\left(a(u)\nabla u-u \nabla \mathcal{K} \ast u\right) \cdot \overrightarrow{\textbf{n}}=0, \;\; &&\quad \text{on } \Sigma_T:=\partial \mathcal{O} \times (0,T).
		\end{split}
	\end{cases}
\end{equation}
Here, we keep Assumptions \ref{a1}-\ref{a3}. Additionally, we denote by $\left(\mathcal{W}_t\right)_{t \geq 0}$ is a cylindrical Wiener process with a noise intensity $\sigma : \mathbb{R} \longrightarrow \mathbb{R}$ which satisfies certain conditions that will be specified in the sequel.

We arrange the rest of this paper as follows. In Section \ref{section2}, we introduce some basic definitions in relation to the stochastic integral with respect to cylindrical Wiener processes, we define the notion of weak martingale solutions and we state the main result of the paper, whose proof is arranged thereafter as follows. In Section \ref{section3}, we consider a non-degenerate stochastic system and establish its associated sequence of almost-surely positive essentially-bounded Faedo-Galerkin solutions. Then, we proceed to derive additional properties satisfied by the sequence according to the following manner: we devote Section \ref{section4} to uniform a priori and temporal translation estimates; Section \ref{section5}, to the tightness of the probability laws and Skorokhod's representation; and Section \ref{section7} to the retrieval of weak martingale solutions to Model \eqref{basicstoch} by passage to the limit and the establishment of a uniqueness result based on a duality approach. In Section  \ref{section8}, we provide a few numerical simulations to illustrate the effect of environmental noise and identify further open questions which are briefly discussed in Section \ref{section9}.
\section{Preliminaries, notion of weak martingale solutions and main result}\label{section2}
Let $\left(\Omega,\mathcal{F},\left(\mathcal{F}_t\right)_{t \geq0},\mathbb{P}\right)$  be a filtered probabilistic space such that the filtration $\left(\mathcal{F}_t\right)_{t \geq0}$ satisfies the usual conditions. Given a separable Banach space $\left(\mathbb{B}, \Vert \,.\, \Vert_{\mathbb{B}}\right)$, we denote by $\mathcal{B}(\mathbb{B})$ the Borel $\sigma$-algebra in $\mathbb{B}$. A $\mathbb{B}$-valued random variable is a mapping $ X : \left(\Omega,\mathcal{F}\right) \longrightarrow \left(\mathbb{B},\mathcal{B}(\mathbb{B})\right)$ which is measurable. Given $p \geq 1$, $\mathcal{L}^p\left(\Omega,\mathcal{F},\mathbb{P};\mathbb{B}\right)\overset{\Delta}{=}\mathcal{L}^p\left(\Omega;\mathbb{B}\right)$ denotes the Banach space consisting of all $\mathbb{B}$-valued random variables, equipped with the following (well-defined) norms:
$$
\Vert X \Vert_{\mathcal{L}^p\left(\Omega;\mathbb{B}\right)}:=\left(\mathbb{E} \Vert X \Vert_{\mathbb{B}}^p\right)^{\frac{1}{p}}, \quad \forall X \in \mathcal{L}^p\left(\Omega;\mathbb{B}\right), \quad \forall 1 \leq p < \infty,
$$
and
$$
\Vert X \Vert_{\mathcal{L}^\infty\left(\Omega;\mathbb{B}\right)}:=\underset{\omega \in \Omega}{\sup} \Vert X(\omega) \Vert_{\mathbb{B}}.
$$
A $\mathbb{B}$-valued stochastic process is defined as a collection of random variables $X:=\left(X_t\right)_{0 \leq t \leq T}$. Then, $X$ is said to be measurable if the mapping $X : \left(\Omega \times [0,T],\mathcal{F} \times \mathcal{B}([0,T]) \right)  \longrightarrow  \left(\mathbb{B},\mathcal{B}(\mathbb{B})\right)$
is measurable. On the other hand, $X$ is said to be adapted if $X_t$ is $\mathcal{F}_t$-measurable for all $t \in [0,T]$. 

Before proceeding to define the stochastic integral in a given separable Banach space $\mathbb{B}$, as a first step, one needs to define the type of stochastic processes serving as integrands. The first requirement is predictability. We recall that a stochastic process is said to be predictable if it is measurable as a mapping $X : \left(\Omega \times [0,T],\mathcal{P}_T \times \mathcal{B}([0,T]) \right)  \longrightarrow  \left(\mathbb{B},\mathcal{B}(\mathbb{B})\right)$, where $\mathcal{P}_T$ denotes the $\sigma$-algebra generated by all left-continuous adapted processes in $\Omega \times [0,T].$ Moving on, as a second step, one needs to define the stochastic process with respect to which the integration is performed, which is in this case, the cylindrical Wiener process. To this end, let $\mathbb{H}$ be a Hilbert space with a complete orthonormal basis $\left(\varphi_k\right)_{k \in \mathbb{N}^*}$. A cylindrical Wiener process $\left(\mathcal{W}_t\right)_{0 \leq t \leq T}$ in $\mathbb{H}$ is defined by 
\begin{equation}\label{sum}
\mathcal{W}_t:= \displaystyle \sum_{k \in \mathbb{N}^*} W_k(t) \varphi_k, \quad \forall t \in [0,T],
\end{equation}
where $\left\{\left(W_k(t)\right)_{t \geq 0}, k \in \mathbb{N}\right\}$ is a sequence of independent $\mathcal{F}_t$-adapted one-dimensional Wiener processes. Actually, the series in Equality \eqref{sum} does not converge in the Hilbert space $\mathbb{H}$. That being said, one can always find a larger separable Hilbert space $\mathbb{H}_0$ for which the canonical linear embedding $\mathcal{I} : \mathbb{H} \longrightarrow \mathbb{H}_0$ has a finite Hilbert-Schmidt norm \cite{da2014stochastic}. This consequently yields the convergence of the series in Equality \eqref{sum}. 
Thus, in addition to imposing predictability, as second requirement, one imposes the integrands to belong to the Banach space $\mathcal{L}^2(\Omega;L_2(\mathbb{H};\mathbb{X}))$, where $L_2(\mathbb{H};\mathbb{X})$ stands for the  space of Hilbert-Schmidt operators. For a given integrand $\sigma$ satisfying the previously-mentioned requirements, one defines the stochastic integral with respect to a cylindrical Wiener process $\left(\mathcal{W}_t\right)_{0 \leq t \leq T}$ as follows: 
$$
\displaystyle \int_0^t \sigma_s d\mathcal{W}_s:=\sum_{k \in \mathbb{N}^*} \int_0^t \sigma_{s,k} dW_k(s),
$$
where 
\begin{equation}\label{nt}
\sigma_{s,k}:=\sigma_s \varphi_k,\; \forall s \in [0,T],\; \forall k \in \mathbb{N}^*.
\end{equation}

Prior to defining the notion of weak martingale solutions, let us first introduce some standard functional spaces which will be used throughout the paper. Namely, given $N \in \{1,2,3\}$ by $L^2(\mathcal{O})^N$, we denote the standard (product) Lebesgue space, equipped with a scalar product which for simplicity will be denoted $\langle, \rangle$ independently of $N$. Additionally, by $H^1(\mathcal{O})$ and $H^2(\mathcal{O})$, we denote the first and second order standard Sobolev spaces, respectively. On the other hand, for a Banach space $\left(\mathscr{X},\Vert \,.\,\Vert_{\mathscr{X}}\right)$, by $L^2(0,T;\mathscr{X})$, we denote the Lebesgue-Bochner space. Finally, by $C([0,T];\mathscr{X})$ we denote the Banach space of continuous $\mathscr{X}$-valued functions equipped with the uniform  norm. For more details on the aforementioned spaces, one can consult for instance \cite{adams2003sobolev}.

As far as the noise coefficient is concerned, we assume it to be an operator $\sigma : L^2(\mathcal{O}) \longrightarrow L_2(\mathbb{H};L^2(\mathcal{O}))$ for which the following assumptions are fulfilled: 
\begin{enumerate}[label={({A${{_\arabic*}}$})}]
	\setcounter{enumi}{3}
	\item \label{a4} \textbf{Lipschitz condition}: 
	$$\exists L>0,\; \forall u,v \in L^2(\mathcal{O}),\; \Vert \sigma(u)-\sigma(v) \Vert_{L_2(\mathbb{H};L^2(\mathcal{O}))} \leq L \Vert u-v \Vert_{L^2(\mathcal{O})};$$ 
	\item \label{a5} \textbf{Growth condition}:
	$$\exists C>0\; \forall u \in L^2(\mathcal{O}),\; \Vert \sigma(u) \Vert_{L_2(\mathbb{H};L^2(\mathcal{O}))} \leq C_1 \left(1+\Vert u\Vert_{L^2(\mathcal{O})}\right).$$
\end{enumerate}
\begin{rem}
Consider a Lipchitz real-valued function $\hat{\sigma} : \mathbb{R} \longrightarrow \mathbb{R}$ with a growth condition of the type 
$$
\exists C_2>0,\; \forall u \in \mathbb{R},\; \vert \hat{\sigma}(u) \vert \leq C_2 \left(1+\vert u \vert\right).$$
Additionally, set
$\sigma :  L^2(\mathcal{O}) \ni u \mapsto \sigma(u) \in L_2(\mathbb{H};L^2(\mathcal{O})),$ where $\sigma(u)\varphi_k=\hat{\sigma}(u(.)) \varphi_k$, where $\left(\varphi_k\right)_{k \in \mathbb{N}^*}$ is the orthonormal basis corresponding to $\mathbb{H}$. Then, it follows that $\sigma$ fulfills Assumptions \ref{a4} and \ref{a5}. Needless to say, various choices can be assigned to $\hat{\sigma}$. 
\end{rem}

With the above setting taken into account, we introduce  the  following definition.
\begin{defi}\label{dfwk}
Given a probability measure $\mu$ on $L^2(\mathcal{O})$, a weak martingale solution to Model \eqref{basicstoch} is defined as a couple $\left(\mathscr{S},u\right)$ satisfying the following conditions:
\begin{enumerate}
	\item $\mathscr{S}:=\left(\Omega,\mathcal{F},\{\mathcal{F}_t\}_{0 \leq t \leq T},\mathbb{P},\mathcal{W}\right)$ is a stochastic basis;
	\item $\mathcal{W}:=\left(\displaystyle \sum_{k \in \mathbb{N}^*}  W_k(t) \varphi_k\right)_{0 \leq t \leq T}$ is an $\{\mathcal{F}_t\}_{0 \leq t \leq T}$-adapted cylindrical Brownian motion;
	\item $\mathbb{P}$-almost-surely ($\mathbb{P}$-a.s) $\forall \omega \in \Omega,$ $u(\omega) \in L^\infty(Q_T) \cap C([0,T];L^2(\mathcal{O}))$ and ${A}(u(\omega)) \in L^2(0,T;H^1(\mathcal{O}))$, where  
	$$
	{A}(u(\omega))(x,t):=\displaystyle \int_{0}^{u(\omega)(x,t)} a(r) dr, \quad \forall x \in \mathcal{O},\; \forall (x,t) \in Q_T;
	$$
	\item The probabilistic law of the initial condition $u_0$ is given by 
	$
    \mathbb{P} \circ u_0^{-1}=:\mu;
	$
	\item The following identity holds $\mathbb{P}$-a.s:
\begin{equation}\label{idn}
	\begin{split}
	\int_\mathcal{O} u(t)v\; dx&= \int_\mathcal{O} u_0v\; dx- \int_{0}^t \int_{\mathcal{O}} a(u(s))\nabla u(s) \cdot \nabla v\; dxdt+\int_{0}^t \int_{\mathcal{O}} u(s)\left(\nabla \mathcal{K} \ast u(s)\right) \cdot \nabla v\; dxdt\\
	&+\int_{0}^t \int_{\mathcal{O}} f(u(s))v\; dxdt+\int_{0}^t \int_{\mathcal{O}} \sigma(u(s))v\; dx d\mathcal{W}_s, \quad \forall v \in H^1(\mathcal{O}), \; \forall t \in [0,T].
	\end{split}
\end{equation}
\end{enumerate}
\end{defi}
\begin{rem}\label{rmbdg}
Note that the last term on the right-hand side of Identity \eqref{idn} is well-defined as a direct consequence of the stochastic Fubini theorem, Assumption \ref{a5} along with the following well-known Burkholder-Davis-Gundy (BDG) inequality \cite{revuz2013continuous}:
$$
\mathbb{E}\left(\sup _{t \in[0, T]}\left\|\int_0^t \sigma(u(s)) d \mathcal{W}_s\right\|_{L^2(\mathcal{O})}^p\right) \leq C_3\; \mathbb{E}\left(\int_0^T\| \sigma(u(t))\|_{{L}_2(\mathbb{H} ;L^2(\mathcal{O}))}^2 dt\right)^{\frac{p}{2}}, \quad \forall p \geq 1,
$$
where $\mathbb{E}$ denotes the mathematical expectation in $(\Omega,\mathcal{F},\mathbb{P})$ and $C_3>0$.
\end{rem}
Following the preceding preliminaries and definitions, our main result can be stated as follows. 
\begin{thm}
Let Assumptions \ref{a1}-\ref{a5} be satisfied. Assume further that $0 \leq u_0 \leq \overline{u},\; \mathbb{P}\text{-a.s},$  such that the (initial) probability measure $\mu_{u_0}$ on $L^2(\mathcal{O})$ satisfies the following moment
conditition
$$
\exists q>2, \quad \int_{L^2(\mathcal{O})} \Vert z \Vert_{L^2(\mathcal{O})}^q d\mu_{u_0}(z)< \infty,
$$
and $\sigma(0)=\sigma(\overline{u})=0$. Then, it holds that Model \eqref{basicstoch} has a weak almost-surely positive martingale solution. Moreover, if $(u_1,\mathcal{W})$ and $(u_2,\mathcal{W})$ are two weak martingale solutions with the same initial condition, then
	$$\tilde{u}_1=\tilde{u}_2 \; \tilde{\mathbb{P}}\text{-a.s.} \; \text{in } {Q}_T.$$
\end{thm}
\section{Construction of a sequence of Faedo-Galerkin solutions to a non-degenerate system}\label{section3}
\subsection{The non-degenerate system}
The double-sided degeneracy of the nonlinear diffusion coefficient causes a major difficulty in the mathematical analysis of Model \eqref{basicstoch} since the non-degenerate assumption of the type: 
$$
\exists \underline{a}>0,\; a(u)>\underline{a}, \; \forall u \in \mathbb{R}^+,
$$
is lost, causing the lack of applicability of standard parabolic theory. One of the ways to get around this difficulty is to proceed by a regularization technique by first considering the following sequence of non-degenerate diffusion coefficients:
$$
a_\varepsilon(u)=a(u)+\varepsilon,\; \forall \varepsilon>0, \forall u \in \mathbb{R}^+.
$$ 
In the deterministic case $(\sigma \equiv 0)$, this technique transforms the difficulty from directly analyzing the nonlinear degenerate problem to deriving suitable a priori estimate on the deterministic sequence of solutions $\left(u_\varepsilon\right)_{\varepsilon>0}$ to the non-degenerate version of the problem, and then using compactness arguments such as Aubin-Lions lemma \cite{simon1986compact}, let $\varepsilon \rightarrow 0$ (in a suitable topology) and retrieve the solution in the degenerate case. That being said, we aim to extend this technique to the case of weak martingale solutions for the class of degenerate stochastic systems considered in this paper. To this end, given $\varepsilon>0$, we consider the following alternative model:
\begin{equation}\label{basicstochnondeg}
	\begin{cases}
		\begin{split}
			&du=\left(\div\left(a_\varepsilon(u) \nabla u-u \nabla \mathcal{K} \ast u\right)+f(u)\right)dt+\sigma(u(t))d\mathcal{W}_t, &&\quad \text{in } Q_T:=\mathcal{O} \times (0,T),\\
			&u(.,0)=u_0,  &&\quad \text{in } \mathcal{O},\\
			&\left(a(u)\nabla u-u \nabla \mathcal{K} \ast u\right) \cdot \overrightarrow{\textbf{n}}=0, \;\; &&\quad \text{on } \Sigma_T:=\partial \mathcal{O} \times (0,T).
		\end{split}
	\end{cases}
\end{equation}
\begin{rem}
Note that unlike the deterministic case $(\sigma \equiv 0)$, achieving a priori estimates on the solution in the non-degenerate case, alone, does not allow to retrieve the solution in the degenerate case. This is primarily due to Definition \ref{dfwk}, which firstly involves additionally finding the stochastic basis $\mathcal{S}$ and, more importantly, does not involve the dual pairing $\langle \partial_t u,v \rangle_{H^1(\mathcal{O})^* \times H^1(\mathcal{O})}$ thus prohibiting directly using Aubin-lions compactness lemma through an $L^2(0,T;H^1(\mathcal{O})^*)$-type estimate as in \cite{bendahmane2023optimal}.
\end{rem}
\subsection{A sequence of Faedo-Galerkin solutions}
Let $\mathcal{S}$ be a fixed stochastic basis and $u_0 \in \mathcal{L}^2(\Omega;L^2(\mathcal{O}))$ be an initial condition which is $\mathcal{F}_0$-measubrable with a given probabilistic law $\mu$. We start by the following eigenvalue problem:
\begin{equation}\label{eigen}
\begin{cases}
	\begin{split}
	&-\Delta u =\lambda u, \quad &&\text{in } \mathcal{O}, \\
	&\partial_\eta u=0, \quad &&\text{on } \partial \mathcal{O},
	\end{split}
\end{cases}
\end{equation}
where the unknown is the couple $(\lambda,u)$.

\noindent
Let $H^2_N:=\{u \in H^2(\mathcal{O}): \quad \partial_\eta u=0 \text{ on } \partial \mathcal{O}\}.$ By classical elliptic theory, there exists an increasing sequence of eigenvalues $\left(\lambda_k\right)_{k \in \mathbb{N}^*}  \subset \mathbb{R}^+$ and eigenvectors $\left(\ell_k\right)_{k \in \mathbb{N}^*} \subset H^2_N(\mathcal{O})$ for which Problem \eqref{eigen} is satisfied. Additionally, $\left(\ell_k\right)_{k \in \mathbb{N}^*}$ has a regularity which increases in terms of the smoothness of $\mathcal{O}$. Moreover, $\left(\ell_k\right)_{k \in \mathbb{N}^*}$ forms an orthonormal basis in $L^2(\mathcal{O})$. We refer for example to \cite[Section 4]{bendahmane2022martingale} and \cite[Section 3]{garcke2016global} for more details on such results. That being said, we denote by $\Pi_n$ the orthogonal projection from $L^2(\mathcal{O})$ equipped with its usual scalar product $\langle,\rangle$ onto $\Sigma_n:=span\{\ell_1,\cdots,\ell_n\}$ with
$$
\Pi_n u:= \sum_{i=1}^n \langle u, \ell_k \rangle \ell_k, \quad \forall u \in L^2(\mathcal{O}), \forall n \in \mathbb{N}^*.
$$
We now proceed to construct the sequence of Faedo-Galerkin solutions. To this end, let $n \in \mathbb{N}^*$, we define the following approximation:
$$
u^n(t):=\sum_{k=1}^n  c_k^n(t) \ell_k,
$$
with the initial condition 
$$
u^n_0:=\sum_{k=1}^n \langle u_0,\ell_k \rangle \ell_k.
$$
On the other hand, we also approximate the noise coefficient by
$$
\sigma^n_k(u^n(t)):=\sum_{l=1}^n  \langle \sigma_{k}(u^n(t)),\ell_l \rangle  \ell_l, \quad  \forall  k \in \mathbb{N}^* \quad (\textrm{see Formula } \eqref{nt}),
$$
where $\left(c_k^n\right)_{k \in \{1,\cdots,n\}}$ is a finite sequence of functions to be determined such that the following equation holds for all $k \in \{1,\cdots,n\}$, $n\in \mathbb{N}^*$:
\begin{equation}\label{galrk}
	\begin{split}
	\langle u^n(t),\ell_k \rangle&=\langle u^n_0,\ell_k \rangle-\int_0^t \langle a_{\varepsilon_n}(u^n(s)) \nabla u(s), \nabla \ell_k \rangle ds+\int_0^t\langle u^n(s)\left(\nabla \mathcal{K} \ast u^n(s)\right), \nabla \ell_k \rangle ds\\
	&+\int_0^t\langle f_M(u^n(s),\ell_k)\rangle ds+\sum_{l=1}^n \int_{0}^t \langle \sigma_{l}^n(u^n(s)),\ell_k\rangle d\mathcal{W}^l_s, \quad \forall t \in [0,T].
	\end{split}
\end{equation}
\noindent
Herein,

{
$$
\varepsilon_n:=\dfrac{1}{n}, \quad \forall n \in \mathbb{N}^*,
$$
}
and
\begin{equation}
	\label{trunc}
f_M(u):=
\begin{cases} 
\begin{split}
&f(u), \quad &&\text{if } \vert u \vert \leq M,\\
&M \quad &&\text{if } u> M,\\
&-M \quad &&\text{if } u<-M,
\end{split}
\end{cases}
\end{equation}
for a given constant $M>0$ which will be specified thereafter accordingly.

\begin{rem}
The truncation function $f_M$ given by Formula \eqref{trunc} will be used in order to guarantee the local monotonicity condition required for the existence of global solutions to the constructed Faedo-Galerkin problem, as we will explore in the upcoming lemma. 
\end{rem}
\noindent
Owing to the definition of $\Pi_n$, the regularity $u^n(t) \in H^2_N(\mathcal{O}),\; \forall t \in [0,T]$ and the embedding $H^2_N(\mathcal{O}) \hookrightarrow L^\infty(\mathcal{O})$ for a spatial dimension $d\in \{1,2,3\}$, Equation \eqref{galrk} can be rewritten as
\begin{equation}\label{galrkmod}
	\begin{split}
	u^n(t)&=u^n_0+\int_0^t \Pi_n\left( \div\left(a_{\varepsilon_n}(u^n(s)) \nabla u^n(s)-u^n(s)\nabla \mathcal{K} \ast u^n(s)\right)\right)ds+\int_0^t\Pi_n\left( f_M(u^n(s))\right)ds\\
	&+\int_{0}^t \sigma^n(u^n(s))   d\mathcal{W}^n_s, \quad \forall t \in [0,T],
	\end{split}
\end{equation}
which can be seen as a stochastic equation in $\mathbb{R}^n$, with $$\sigma^n(u^n(s))   d\mathcal{W}^n_s\overset{\Delta}{=}\sum_{l=1}^n \sigma^n_l(u^n(t)) d\mathcal{W}^l_s, \quad \forall s \in [0,T].$$

\noindent
Next, we proceed to justify the existence of solutions to Equation \eqref{galrkmod}. More precisely, we establish the following lemma.
\begin{lem}\label{éxisfaédo}
Under Assumptions \ref{a1}-\ref{a5}, it holds that Equation \eqref{galrkmod} has a unique $\{\mathcal{F}_t\}_{0 \leq t \leq T}$-adapted strong solution $\left(u^n(t)\right)_{0 \leq t \leq T}$. Moreover, $u^n \in C([0,T];\Sigma_n)$.
\end{lem}
\begin{proof}
It is clear that Equation \eqref{galrkmod} can be rewritten as a particular case of  the following class of stochastic differential equations:
\begin{equation}
	\begin{cases}
	du(t)=\mathscr{F}(u(t))dt+\mathscr{G}(u(t))d\mathcal{W}^n_s,\\
	u(0)=u^n_0,
	\end{cases}
\end{equation}
which are known to have unique $\{\mathcal{F}_t\}_{0 \leq t \leq T}$-adapted strong solutions provided that the following assumptions are fulfilled (see for example \cite[Theorem 1.2]{krylov1999kolmogorov}): 
\begin{enumerate}
	\item \textbf{Weak coercivity:} $\exists C_4>0, \; \forall u \in \Sigma_n,\quad
	 2 \langle \mathscr{F}(u),u\rangle+ \Vert \mathscr{G}(u) \Vert_{L_2(\mathbb{H};\Sigma_n)}^2 \leq C_4 \left(1+\Vert u \Vert_{\Sigma_n}^2\right)$;
	 \item \textbf{Local weak monotonicity:}  $\forall R>0,\; \exists L(R)>0$ such that $$
	 \underset{i \in \{1,2\}}{\max} \Vert u_i \Vert_{L^2(\mathcal{O})} \leq R \implies 2 \langle \mathscr{F}(u_1)-\mathscr{F}(u_2),u_1-u_2\rangle+ \Vert \mathscr{G}(u_1)-\mathscr{G}(u_2) \Vert_{L_2(\mathbb{H};\Sigma_n)}^2 \leq L(R) \left(\Vert u_1-u_2 \Vert_{\Sigma_n}^2\right).$$
\end{enumerate}
In our case 
$$
\mathscr{F}(u):=\Pi_n\left(\div\left(a_{\varepsilon_n}(u) \nabla u-u\nabla \mathcal{K} \ast u\right)\right)+\Pi_n\left( f_M(u)\right),
$$
and 
$$
\mathscr{G}(u):=\sigma^n(u).
$$
As far as the weak coercivity is concerned, by using the regularity of the orthonormal basis $\left(\ell_k\right)_{k \in \mathbb{N}^*}$, a straightforward calculation and the use of Cauchy-Schwarz inequality allow us to deduce that
\begin{equation}\label{part1}
\Vert \Pi_n\left(\div\left(a_{\varepsilon_n}(u) \nabla u\right)\right)\Vert_{L^2(\mathcal{O})}\leq \sum_{k=1}^n \vert \langle a_{\varepsilon_n}(u) \nabla u,\nabla \ell_k\rangle \vert \Vert \ell_k \Vert_{L^2(\mathcal{O})} \leq C_{\overline{u}} \Vert \nabla u \Vert_{L^2(\mathcal{O})^d} \Vert \nabla \ell_k \Vert_{L^2(\mathcal{O})^d} \Vert \ell_k \Vert_{L^2(\mathcal{O})}.
\end{equation}
By recalling that $u \in \Sigma_n$ and taking into account the orthonormal property of $\left(\ell_k\right)_{k \in \mathbb{N}^*}$, it follows from Inequality \eqref{part1} that
\begin{equation}\label{diva}
	\Vert \Pi_n\left(\div\left(a_{\varepsilon_n}(u) \nabla u\right)\right)\Vert_{L^2(\mathcal{O})}\leq C_{5} \Vert u \Vert_{\Sigma_n}.
\end{equation}
By the same previous analogy and taking  Assumption \ref{a2} into consideration, we obtain that  
\begin{equation}\label{divK}
	\Vert \Pi_n\left(\div\left(-u\nabla \mathcal{K} \ast u\right)\right)\Vert_{L^2(\mathcal{O})}\leq C_{6} \Vert u \Vert_{\Sigma_n}.
\end{equation}
On the other hand, by the global Lipschitz property of the truncation $f_M$, we obtain
\begin{equation}\label{fm}
	\Vert \Pi_n\left(f_M(u)\right)\Vert_{L^2(\mathcal{O})}\leq C_{7} \Vert u \Vert_{\Sigma_n},
\end{equation}
where $C_5$, $C_6$ and $C_7$ are positive constants depending on $n$.

\noindent
Thus, Inequalities \eqref{diva}-\eqref{fm}, Cauchy-Schwarz inequality along with Assumption \ref{a5} yield the weak coercivity.

Let us now establish the local weak monotonicity. For this purpose, let $R>0$ and $u_1,u_2 \in L^2(\mathcal{O})$ such that $\underset{i \in \{1,2\}}{\max} \Vert u_i \Vert_{L^2(\mathcal{O})} \leq R$. Set 
$$
\mathscr{D}_a:=\Pi_n\left(\div\left(a_{\varepsilon_n}(u_1) \nabla u_1\right)\right)-\Pi_n\left(\div\left(a_{\varepsilon_n}(u_2) \nabla u_2\right)\right),
$$
$$
\mathscr{D}_\mathcal{K}:=\Pi_n\left(\div\left(-u_1\nabla \mathcal{K} \ast u_1\right)\right)-\Pi_n\left(\div\left(-u_2\nabla \mathcal{K} \ast u_2\right)\right),
$$
and 
$$
\mathscr{D}_f:=\Pi_n\left( f_M(u_1)\right)-\Pi_n\left(f_M(u_2)\right).
$$
Hence, using Assumption \ref{a1} and the definition of $\Pi_n$, there exists a constant $C_8>0$ depending on $n$ such that
\begin{equation}\label{d1part1}
\Vert \mathscr{D}_a \Vert_{L^2(\mathcal{O})} \leq C_8\left(\Vert a_{\varepsilon_n} (u_1)-a_{\varepsilon_n}(u_2) \Vert_{L^2(\mathcal{O})} \Vert \nabla u_1 \Vert_{L^2(\mathcal{O})^d}+\Vert \nabla u_1-\nabla u_2 \Vert_{L^2(\mathcal{O})^d}\right).
\end{equation}
By using Assumption \ref{a1} once again with the mean value theorem, there exists $\alpha \in (0,1)$ such that
$$
a_{\varepsilon_n}(u_1)-a_{\varepsilon_n}(u_2)=a_{\varepsilon_n}^\prime(\left(1-\alpha\right)u_1+\alpha u_2)(u_1-u_2).
$$
Recalling the fact $u_1-u_2 \in \Sigma_n$ and  $\left(\ell_k\right)_{k \in \mathbb{N}^*} \subset H_N^2(\mathcal{O}) \hookrightarrow L^\infty(\mathcal{O})$ in a spatial dimension $d \in \{1,2,3\}$, it follows that $\left(1-\alpha\right)u_1+\alpha u_2 \in L^\infty(\mathcal{O})$ which by the local Lipschitz property of $a_{\varepsilon_n}$ allows us to obtain from Inequality \eqref{d1part1} that
\begin{equation}\label{d1part2}
	\Vert \mathscr{D}_a \Vert_{L^2(\mathcal{O})} \leq C_9\left(\Vert u_1-u_2 \Vert_{\Sigma_n}\right),
\end{equation}
where $C_9>0$ depends on $a$, $n$ and $R$.

\noindent
By the same arguments used for $\mathcal{D}_a$ and taking into account this time Assumption \ref{a2}, we eventually obtain
\begin{equation}\label{dk}
	\Vert \mathscr{D}_{\mathcal{K}} \Vert_{L^2(\mathcal{O})} \leq C_{10}\left(\Vert u_1-u_2 \Vert_{\Sigma_n}\right),
\end{equation}
where $C_{10}>0$ depends on $\mathcal{K}$, $n$ and $R$.
\noindent
On the other hand, the global Lipschitz property of the truncation $f_M$ allows us to obtain 
\begin{equation}\label{df}
	\Vert \mathscr{D}_{f} \Vert_{L^2(\mathcal{O})} \leq C_{11}\left(\Vert u_1-u_2 \Vert_{\Sigma_n}\right),
\end{equation}
where $C_{11}>0$ depends on $M>0$.

\noindent
Finally, combining Inequalities \eqref{d1part2}-\eqref{df} along with Assumption \ref{a5} and using Cauchy-Schwarz inequality yield the local weak monotonicity.

The proof is thus concluded.
\end{proof}
\subsection{Almost-sure positiveness and boundedness of the sequence of Faedo-Galerkin solutions}
In this section, we proceed to establish two main properties of the sequence of Faedo-Galerkin solutions: almost-sure positiveness and almost-sure boundedness. These properties will be essential to establish the desired uniform a priori estimates in Section \ref{section5}. To this end, we rely on the so-called Stampachia approach used in \cite{chekroun2016stampacchia} and \cite{bendahmane2022martingale}.
\begin{lem}\label{lemmaa}
	Let Assumptions \ref{a1}-\ref{a5} be satisfied and assume further that $0 \leq u_0 \leq \overline{u},\; \mathbb{P}\text{-a.s},$ where $\overline{u}$ is as defined in Assumption \ref{a1}. Then, the following implications hold:
\begin{equation}\label{firstimpli}
\sigma(0)=0 \implies 0 \leq u^n(t), \quad \mathbb{P}\text{-a.s}, \; \forall t \in (0,T),\; \forall \varepsilon>0,\; \forall n \in \mathbb{N}^*.
\end{equation}
\begin{equation}\label{sécoondéimpli}
\sigma(\overline{u})=0 	\implies u^n(t) \leq \overline{u}, \quad \mathbb{P}\text{-a.s}, \; \forall t \in (0,T),\; \forall \varepsilon>0,\; \forall n \in \mathbb{N}^*.
\end{equation}
\end{lem}
\begin{proof}
Let $\varepsilon>0$ and $n \in \mathbb{N}^*$. Under the notations and definitions of the preceding subsection, we begin by considering the following first intermediate sequence of Faedo-Galerkin solutions generated by
\begin{equation}\label{intrmgalrkmod}
	\begin{split}
		u^n(t)&=u^n_0+\int_0^t \Pi_n\left( \div\left(a_{\varepsilon_n}(\left(u^n(s)\right)^+) \nabla u^n(s)-\left(u^n(s)\right)^+\nabla \mathcal{K} \ast u^n(s)\right)\right)ds+\int_0^t\Pi_n\left( f_M(\left(u^n(s)\right)^+)\right)ds\\
		&+\int_{0}^t \sigma^n(u^n(s))   d\mathcal{W}^n_s, \quad \forall t \in [0,T],
	\end{split}
\end{equation}
where $u^+:=\max\{0,u\}, \; \forall u \in \mathbb{R},$ stands for the positive part.
\noindent
We further consider the following regularization of the negative part (see e.g. \cite[Section 9]{bendahmane2022martingale}): 
\begin{equation}
		\mathcal{R}_\nu(u):=\begin{cases}	
			\begin{split}
			&u^2- \dfrac{\nu^2}{6}, &&\text{ if } u<-\nu,\\
			&- \dfrac{u^4}{2 \nu^2}-\dfrac{4 u^3}{3 \nu}, &&\text{ if } -\nu \leq u<0,\\
			&0, &&\text{ if } u\geq 0,
	\end{split}	
\end{cases}	
\end{equation}
so that 
$$
\mathcal{R}_{\nu}^{\prime}(u)=\left\{\begin{array}{ll}
	2 u &\text{ if } u<-\nu, \\
	-\frac{2 u^3}{\nu^2}-\frac{4 u^2}{\nu} & \text{ if } -\nu \leq u<0,\\
	0 & \text{ if } u\geq 0,
\end{array} \quad \mathcal{R}_{\nu}^{\prime \prime}(u)= \begin{cases}2 &\text{ if } u<-\nu,\\
	-\frac{6 u^2}{\nu^2}-\frac{8 u}{\nu} & \text{ if } -\nu \leq u<0, \\
	0 & \text{ if } u\geq 0,\end{cases}\right.$$
$$\mathcal{R}_{\nu}(u) \geq 0,\; \mathcal{R}_{\nu}^{\prime}(u) \leq 0,\; \mathcal{R}_{\nu}^{\prime \prime}(u) \geq 0,\; \forall u \in \mathbb{R},$$ 
and uniformly in $u \in \mathbb{R}$ as $\nu \rightarrow 0$, it holds that  $$\mathcal{R}_{\nu}(u) \rightarrow\left(u^{-}\right)^2,$$ $$\mathcal{R}_{\nu}^{\prime}(u) \rightarrow-2 u^{-},$$ and $$\mathcal{R}_{\nu}^{\prime \prime}(u) \rightarrow\left\{\begin{array}{ll}2 & \text { if } u<0, \\ 0 & \text { if } u \geq 0,\end{array}\right.$$
where $u^-:=-\max\{0,-u\}, \; \forall u \in \mathbb{R},$ stands for the negative part.

\noindent
Let $t\in (0,T)$. We apply Itô's and Green's formulas to compute $\displaystyle \int_\mathcal{O} \mathcal{R}_{\nu}\left(u^n(t)\right) dx$ as follows:
\begin{align}\label{itorég}
\int_\mathcal{O} \mathcal{R}_{\nu}\left(u^n(t)\right) dx&=\int_\mathcal{O} \mathcal{R}_{\nu}\left(u^n_0\right) dx \nonumber-\int_0^t \int_\mathcal{O} {\mathcal{R}_\nu}^{\prime \prime}(u^n(s)) a_{\varepsilon_n}(\left(u^n(s)\right)^+) \vert \nabla u^n(s) \vert^2 dx ds \nonumber\\
&+\int_0^t \int_\mathcal{O} {\mathcal{R}_\nu}^{\prime \prime}(u^n(s)) \left(u^n(s)\right)^+\nabla \mathcal{K} \ast u^n(s) \cdot \nabla u^n(s)dxds \nonumber\\
&+\int_0^t \int_{\mathcal{O}} f_M((u^n(s))^+) {\mathcal{R}_\nu}^{\prime}(u^n(s))dx ds\nonumber\\
&+\sum_{k=1}^{n} \int_{0}^t \int_{\mathcal{O}} {\mathcal{R}_\nu}^{\prime}(u^n(s)) \sigma_{k}^n(u^n(s)) dx d\mathcal{W}^n_s\nonumber \\
&+\dfrac{1}{2} \sum_{k=1}^{n} \int_{0}^{t} \int_{\mathcal{O}} {\mathcal{R}_\nu}^{\prime \prime}(u^n(s))  \left(\sigma_{k}^n(u^n(s))\right)^2 dx ds.
\end{align}

\noindent
Now, recalling the definitions of ${\mathcal{R}_\nu}^{\prime}$, ${\mathcal{R}_\nu}^{\prime \prime}$ and $f_M$ and separating the regions $\{u^n \geq 0\}$ and $\{u^n<0\}$,  it follows that
$$-\int_0^t \int_\mathcal{O} {\mathcal{R}_\nu}^{\prime \prime}(u^n(s)) a_{\varepsilon_n}(\left(u^n(s)\right)^+) \vert \nabla u^n(s) \vert^2 dx ds\leq  0,$$
$$\int_0^t \int_\mathcal{O} {\mathcal{R}_\nu}^{\prime \prime}(u^n(s)) \left(u^n(s)\right)^+\nabla \mathcal{K} \ast u^n(s) \cdot \nabla u^n(s)dxds=0,$$
and 
$$\int_0^t \int_{\mathcal{O}} f_M((u^n(s))^+) {\mathcal{R}_\nu}^{\prime}(u^n(s))dxds=0.$$
Consequently, by evaluating the expectation on both sides of Identity \eqref{itorég}, we arrive at
\begin{equation}\label{itorgconsq}
	\mathbb{E} \int_\mathcal{O} \mathcal{R}_{\nu}\left(u^n(t)\right) dx\leq \mathbb{E} \int_\mathcal{O} \mathcal{R}_{\nu}\left(u^n_0\right) dx \nonumber+\dfrac{1}{2} \mathbb{E} \sum_{k=1}^{n} \int_{0}^{t} \int_{\mathcal{O}} {\mathcal{R}_\nu}^{\prime \prime}(u^n(s))  \left(\sigma_{k}^n(u^n(s))\right)^2 dx ds.
\end{equation}
Thus, letting $\nu \rightarrow 0$ and proceeding analogously to \cite[Section 3.2]{chekroun2016stampacchia} while recalling that $\sigma(0)=0$, we obtain
\begin{equation}\label{itorgconsqlim}
	\mathbb{E} \int_\mathcal{O}\left( \left(u^n(t)\right)^-\right)^2 dx\leq \mathbb{E} \int_\mathcal{O} \left({u^n_0}^-\right)^2dx \nonumber+C_{12} \mathbb{E} \sum_{k=1}^{n} \int_{0}^{t} \int_{\mathcal{O}} \left(\left(u^n(s)\right)^-\right)^2  dx ds,
\end{equation}
where $C_{12}>0$.

\noindent
Recalling that $u^n_0 \geq 0\; \mathbb{P}\text{-a.s}$, we use 
Gronwall's inequality to obtain that 
$$u^n(t) \geq 0,\; \mathbb{P}\text{-a.s}, \; \forall t \in (0,T).$$
Hence, the first implication follows by the definition of the sequence generated by Equation \eqref{intrmgalrkmod}.

\noindent

\noindent
To prove the second implication, we consider the second intermediate sequence of Faedo-Galerkin solutions generated by
\begin{equation}\label{intrmgalrkmod2}
	\begin{split}
		v^n(t)&=u^n_0+\int_0^t \Pi_n\left( \div\left(a_{\varepsilon_n}(\left(v^n(s)\right)^+) \nabla v^n(s)-\Psi(v^n(s))\nabla \mathcal{K} \ast v^n(s)\right)\right)ds\\
		&+\int_0^t\Pi_n\left(f_M(v^n(s)) \Psi(v^n(s))\right)ds+\int_{0}^t \sigma^n(v^n(s))   d\mathcal{W}^n_s, \quad \forall t \in [0,T],
	\end{split}
\end{equation}
where 
$
\Psi(v)=v \mathbbm{1}_{\left\{v < \overline{u}\right\}}, \quad \forall  v \in \mathbb{R}.
$

\noindent
Let $t\in (0,T)$. We apply Itô's and Green's formulas to compute $\displaystyle \int_\mathcal{O} \mathcal{R}_{\nu}\left(\overline{u}-v^n(t)\right) dx$ as follows:

\begin{align}\label{itorégv}
	\int_\mathcal{O} \mathcal{R}_{\nu}\left(\overline{u}-v^n(t)\right) dx&=\int_\mathcal{O} \mathcal{R}_{\nu}\left(\overline{u}-u_0^n\right) dx \nonumber-\int_0^t \int_\mathcal{O} {\mathcal{R}_\nu}^{\prime \prime}(\overline{u}-v^n(s)) a_{\varepsilon_n}(v^n(s)) \vert \nabla v^n(s) \vert^2 dx ds \nonumber\\
	&+\int_0^t \int_\mathcal{O} {\mathcal{R}_\nu}^{\prime \prime}(\overline{u}-v^n(s))\Psi(v^n(s))\nabla \mathcal{K} \ast v^n(s) \cdot \nabla v^n(s)dxds \nonumber\\
	&+\int_0^t \int_{\mathcal{O}} f_M(v^n(s)) \Psi(v^n(s)) {\mathcal{R}_\nu}^{\prime}(\overline{u}-v^n(s))dx ds\nonumber\\
	&+\sum_{k=1}^{n} \int_{0}^t \int_{\mathcal{O}} {\mathcal{R}_\nu}^{\prime}(\overline{u}-v^n(s)) \sigma_{k}^n(v^n(s)) dx d\mathcal{W}^n_s\nonumber \\
	&+\dfrac{1}{2} \sum_{k=1}^{n} \int_{0}^{t} \int_{\mathcal{O}} {\mathcal{R}_\nu}^{\prime \prime}(\overline{u}-v^n(s))  \left(\sigma_{k}^n( v^n(s))\right)^2 dx ds.
\end{align}
Now, recalling the definitions of ${\mathcal{R}_\nu}^{\prime}$, ${\mathcal{R}_\nu}^{\prime \prime}$ and $f_M$ and separating the regions $\{v^n \geq \overline{u}\}$ and $\{v^n<\overline{u}\}$, we obtain 
$$-\int_0^t \int_\mathcal{O} {\mathcal{R}_\nu}^{\prime \prime}(\overline{u}-v^n(s)) a_{\varepsilon_n}(\left(v^n(s)\right)^+) \vert \nabla v^n(s) \vert^2 dx ds\leq  0,$$
$$\int_0^t \int_\mathcal{O} {\mathcal{R}_\nu}^{\prime \prime}(\overline{u}-v^n(s))\exp(\kappa t)\Psi(v^n(s))\nabla \mathcal{K} \ast v^n(s) \cdot \nabla v^n(s)dxds=0,$$
\begin{align*}
\int_0^t \int_{\mathcal{O}} f_M(v^n(s)) {\mathcal{R}_\nu}^{\prime}(\overline{u}-v^n(s)) \Psi(v^n(s))dx ds=0,
\end{align*}
consequently
\begin{align}\label{applyéxp}
	\int_\mathcal{O} \mathcal{R}_{\nu}\left(\overline{u}-v^n(t)\right) dx&\leq \int_\mathcal{O} \mathcal{R}_{\nu}\left(\overline{u}-u^n_0\right) dx+\sum_{k=1}^{n} \int_{0}^t \int_{\mathcal{O}} {\mathcal{R}_\nu}^{\prime}(\overline{u}-v^n(s)) \sigma_{k}^n(v^n(s)) dx d\mathcal{W}^n_s\nonumber \\
	&+\dfrac{1}{2} \sum_{k=1}^{n} \int_{0}^{t} \int_{\mathcal{O}} {\mathcal{R}_\nu}^{\prime \prime}(\overline{u}-v^n(s))  \left(\sigma_{k}^n(v^n(s))\right)^2 dx ds. 
\end{align}
Evaluating the expectation on both sides of Inequality \eqref{applyéxp}, sending $\nu \rightarrow 0$ and recalling that $\sigma(\overline{u})=0$,  we obtain
\begin{align}\label{itorégvrés}
	\mathbb{E} \int_\mathcal{O} \left(\left(\overline{u}-v^n(t)\right)^-\right)^2 dx&\leq \mathbb{E} \int_\mathcal{O} \left(\left(\overline{u}-u^n_0\right)^-\right)^2 dx+C_{13} \mathbb{E} \sum_{k=1}^{n} \int_{0}^{t} \int_{\mathcal{O}} \left(\left(\overline{u}-u^n(s)\right)^-\right)^2  dx ds. 
\end{align}
The proof is thus concluded by Gronwall's inequality and the definition of the sequence generated by Equation \eqref{intrmgalrkmod2}.
\end{proof}
In what follows, $C_i\; (i \in \mathbb{N}^*)$ will denote a positive constant which, unless stated otherwise, does not depend on $n$. We also define
\begin{equation}\label{défiA}
A(s):=\int_{0}^{s}  a_{\varepsilon_n}(\theta)d\theta,\;  \forall s \in (0,T).
\end{equation}

\begin{rem}
{Note that the result in Lemma \ref{lemmaa} ensures biological feasibility, as solution remains within $[0,\overline{u}]$, consistent with volume-filling effect.}
\end{rem}
\section{Uniform a priori and temporal translation estimates}\label{section4}
\subsection{Uniform a priori estimates}
First, we let the functions:

$$
A(s):=\int_0^s a(r)\,dr\qquad \text{and}\qquad \mathcal{A}(s):=\int_0^s A(r)\,dr,
$$
and the corresponding nondegenerate functions
$$
A_{\varepsilon_n}(s):=\int_0^s a_{\varepsilon_n}(r)\,dr\qquad \text{and}\qquad \mathcal{A}_{\varepsilon_n}(s):=\int_0^s A_{\varepsilon_n}(r)\,dr,
$$

\noindent The uniform a priori estimate reads as follows.
\begin{lem}\label{lémmaaprioriés}
	Let Assumptions \ref{a1}-\ref{a5} be satisfied and assume further that $0 \leq u_0 \leq \overline{u},\; \mathbb{P}\text{-a.s},$ where $\overline{u}$ is as defined in Assumption \ref{a1} and $\sigma(0)=\sigma(\overline{u})=0$. It holds that the sequence of solutions  $\left(u^n(t)\right)_{0 \leq t \leq T}$ generated by Equation \eqref{galrkmod} satisfies the following a priori estimate: there exists a constant $C>0$ not depending on $n$ such that
\begin{equation}\label{apriori}
\begin{split}
& \mathbb{E} \Vert \mathcal{A}(u^n)(t) \Vert_{L^1(\mathcal{O})}
+\mathbb{E} \Vert \mathcal{A}_{\varepsilon_n}(u^n)(t) \Vert_{L^1(\mathcal{O})}+
 \mathbb{E} \displaystyle \int_0^T \int_{\mathcal{O}} \vert \nabla A(u^n) \vert^2 dxdt+\mathbb{E} \displaystyle \int_0^T \int_{\mathcal{O}} \vert \nabla A_{\varepsilon_n}(u^n) \vert^2 dxdt,\\
&\mathbb{E}  \Vert u^n(t) \Vert_{L^2(\mathcal{O})}^2
+\mathbb{E} \displaystyle \int_0^T \int_{\mathcal{O}} \vert \sqrt{\varepsilon_n}\nabla u^n \vert^2 dxdt \leq C, \; \forall t \in (0,T),
 \end{split}
\end{equation}
and
\begin{equation}\label{apriori-sup}
 \E \Bigl[\sup_{t\in [0,T]}\int_{\mathcal O} \abs{u^n}^2\,dx\Bigl]\leq C.
\end{equation}

\end{lem}

\begin{proof}
{
First, we apply It\^o rule to
$$
\int_{\mathcal{O}}\mathcal{A}(u^n(t))dx,
$$
so the result is 
\begin{equation}\label{firstpart}
\begin{split}
&\int_{\mathcal{O}}\mathcal{A}(u^n(t))\,dx
+ \int_0^t\int_{\mathcal{O}} |\nabla A(u^n(s))|^2 \,dx\,ds+\varepsilon_n  \int_0^t\int_{\mathcal{O}} a(u^n(s))|\nabla u^n(s)|^2 \,dx\,ds\\
&\qquad = \int_{\mathcal{O}}\mathcal{A}(u^n_0)\,dx
+ \int_0^t\int_{\mathcal{O}} u^n(s)\,(\nabla K\ast u^n(s))\cdot \nabla A(u^n(s))\,dx\,ds\\
&\qquad\quad + \int_0^t\int_{\mathcal{O}} f_M(u^n(s))\,A(u^n(s))\,dx\,ds
+ \frac{1}{2}\int_0^t\int_{\mathcal{O}} a(u^n(s))\,\big(\sigma^n(u^n(s))\big)^2 \,dx\,ds\\
&\qquad\quad + \int_0^t\int_{\mathcal{O}} A(u^n(s))\,\sigma_n(u^n(s)) \,dx\,d\mathcal{W}^n_s.
\end{split}
\end{equation}
Observe that 
$$
\varepsilon_n  \int_0^t\int_{\mathcal{O}} a(u^n(s))|\nabla u^n(s)|^2 \,dx\,ds\geq 0.
$$
From Assumptions \ref{a1}, \ref{a3} and \ref{a5} and by using Cauchy-Schwarz and Young inequalities and taking into account Lemma \ref{lemmaa}, it follows that
\begin{equation}\label{young1}
\begin{split}
& \int_0^t\int_{\mathcal{O}} u^n(s)\,(\nabla K\ast u^n(s))\cdot \nabla A_{\varepsilon_n}(u^n(s))\,dx\,ds\\
&\qquad \qquad \qquad \leq
 \dfrac{1}{2}\int_0^t \int_{\mathcal{O}} \vert \nabla A(u^n(s)) \vert^2 dx ds
+\frac{1}{2} \int_0^t\int_{\mathcal{O}} \abs{u^n(s)}^2\abs{\nabla K\ast u^n(s))}^2)\,dx\,ds\\
&\qquad \qquad \qquad \leq
 \dfrac{1}{2}\int_0^t \int_{\mathcal{O}} \vert \nabla A(u^n(s)) \vert^2 dx ds
+C\int_0^t\int_{\mathcal{O}} (1+\abs{u^n(s)}^2)\,dx\,ds,
	\end{split}
\end{equation}
\begin{equation}\label{young2}
\begin{split}
& \int_0^t\int_{\mathcal{O}} f_M(u^n(s))\,A_{\varepsilon_n}(u^n(s))\,dx\,ds \leq 
C\int_0^t\int_{\mathcal{O}} (1+\abs{u^n(s)}^2)\,dx\,ds,
	\end{split}
\end{equation}
and 
\begin{equation}\label{young3}
\begin{split}
&\frac{1}{2}\int_0^t\int_{\mathcal{O}} a_{\varepsilon_n}(u^n(s))\,\big(\sigma^n(u^n(s))\big)^2 \,dx\,ds
\leq C \int_0^t\int_{\mathcal{O}} (1+\abs{u^n(s)}^2)\,dx\,ds,
	\end{split}
\end{equation}
for some constant $C>0$ depending on $L^\infty$-bound of $u^n$.\\
\noindent Note that the martingale property of stochastic integrals
ensures that the expected value of each of the last term in \eqref{firstpart} is zero. Therefore, 
taking the expectation in \eqref{firstpart} and using Gronwall’s inequality, we conclude that
$$
\mathbb{E} \int_{\mathcal{O}}\mathcal{A}(u^n(t))\,dx
+ \mathbb{E} \int_0^t\int_{\mathcal{O}} |\nabla A(u^n(s))|^2 \,dx\,ds \leq C,
$$
for some constant $C>0$. \\
\noindent Similarly, by applying Itô rule to $ \displaystyle \int_{\mathcal{O}}\mathcal{A}_{\varepsilon_n}(u^n(t))dx$, we obain the following energy-type inequality
$$
\mathbb{E} \int_{\mathcal{O}}\mathcal{A}_{\varepsilon_n}(u^n(t))\,dx
+ \mathbb{E} \int_0^t\int_{\mathcal{O}} |\nabla A_{\varepsilon_n}(u^n(s))|^2 \,dx\,ds \leq C,
$$
for some constant $C>0$. \\

Now, we prove the uniform estimate on $\varepsilon_n$-terms. We use It\^o formula for $\displaystyle \frac{1}{2}\int_{\mathcal{O}}|u_n(t)|^2\,dx$ to get
\begin{equation}\begin{split}\label{eq:v+w-tmp2}
\frac{1}{2}\int_{\mathcal{O}}|u_n(t)|^2\,dx &= \frac{1}{2}\int_{\mathcal{O}}|u_n(0)|^2\,dx \\
&\quad +\int_0^t\int_{\mathcal O} u_n\,\nabla\cdot\big(a_{\varepsilon_n}(u_n)\nabla u_n - u_n\nabla K*u_n\big)\,dx\,ds \\
&\quad + \int_0^t\int_{\mathcal O} u_n f_M(u_n)\,dx\,ds \\
&\quad + \frac{1}{2} \int_0^t\int_{\mathcal O} (\sigma_{n}(u_n))^2\,dx\,ds\\
&\quad+ \int_0^t\int_{\mathcal{O}} u^n(s)\,\sigma_n(u^n(s)) \,dx\,d\mathcal{W}^n_s.
\end{split}\end{equation}
An integration by parts (using Neumann boundary conditions) leads to
$$
\int_{\mathcal O} u_n\,\nabla\cdot(a_{\varepsilon_n}(u_n)\nabla u_n)\,dx
= -\int_{\mathcal O} a_{\varepsilon_n}(u_n)|\nabla u_n|^2\,dx.
$$
This implies (recall that $a_{\varepsilon_n}(r)=a(r)+\varepsilon_n\ge \varepsilon_n$)
$$
-\int_{\mathcal O} a_{\varepsilon_n}(u_n)|\nabla u_n|^2\,dx \le -\varepsilon_n\int_{\mathcal O}|\nabla u_n|^2\,dx.
$$
For the nonlocal transport term, we apply again an integration by parts
\begin{align*}
-\int_{\mathcal O} u_n\,\nabla\cdot(u_n\nabla K*u_n)\,dx
&=\int_{\mathcal O} u_n \nabla u_n\cdot\nabla K*u_n\,dx\\
&= - \frac{1}{2}\iint_{\mathcal O\times\mathcal O} u_n^2(x)u_n(y)\,\Delta K(x-y)\,dx\,dy.
\end{align*}
Since $\Delta K\in L^\infty$ and $u_n$ is uniformly bounded in $L^\infty$, we obtain
\begin{align*}
\Bigl|\int_{\mathcal O} u_n\,\nabla\cdot(u_n\nabla K*u_n)\,dx\Bigl| \le \frac{1}{2}\|\Delta K\|_{L^\infty({\mathcal{O}})}\,\|u_n(s)\|_{L^2({\mathcal{O}})}^2\,\|u_n(s)\|_{L^1({\mathcal{O}})}
\le C\|u_n(s)\|_{L^2({\mathcal{O}})}^2,
\end{align*}
for some constant $C>0$ depending on uniform $L^\infty$-bound of $u_n$. 
Regarding the reaction term, we have 
$$
\int_{\mathcal O} u_n f_M(u_n)\,dx \le C(1+\|u_n\|_{L^2({\mathcal{O}})}^2), 
$$
for some constant $C>0$. Moreover the It\^o stochastic term satisfies
$$
\frac{1}{2} \int_{\mathcal{O}} (\sigma_{n}(u_n))^2\le C(1+\|u_n\|_{L^2({\mathcal{O}})}^2).
$$

Note that by the martingale property of stochastic integrals
$$
\mathbb{E} \int_0^t\int_{\mathcal{O}} u^n(s)\,\sigma_n(u^n(s)) \,dx\,d\mathcal{W}^n_s=0.
$$
Collecting all estimates, we obtain
$$
\frac{1}{2}\,\mathbb E\Big[\|u_n(s)\|_{L^2(\mathcal{O})}^2\Big]
+ \varepsilon_n\,\mathbb E\int_0^t\int_{\mathcal O}|\nabla u_n|^2\,dx\,dt
\le C + C\,\mathbb E\int_0^t(1+\|u_n\|_{L^2(\mathcal{O})}^2)\,ds,
$$
An application of Gronwall's lemma yields the second uniform bound \eqref{apriori}.\\
To obtain the estimate \eqref{apriori-sup}, we use the supremum over the interval $[0, T]$ in \eqref{eq:v+w-tmp2},
and then the expectation $\E[\cdot]$. Next, we apply the Burkholder-Davis-Gundy inequality, the 
Cauchy-Schwarz inequality, the assumption on $\sigma$ (A4) to obtain from \eqref{apriori}
\begin{equation}\label{Esup1}
	\begin{split}
		 \E\left[ \sup_{t\in [0,T]} 
		\norm{u^n(t)}_{L^2(\mathcal O)}^2\right]  & \leq S_1\left(1 +
		\E\left[\, \sup_{t\in [0,T]} \abs{\int_0^t\int_{\mathcal O} 
		u^n \sigma_n(u^n) \,dx \,dW^{n}(s)}\, \right]\right)
		\\ &   \leq S_2 \left(1+\E \left[ \left(\int_0^T 
		\sum_{k=1}^n \abs{\int_{\mathcal O} u^n 
		\sigma^n_{k}(u^n) \,dx}^2 \,dt \right)^{\frac12}\right]\right)
		\\ &  \le  S_3 +
		\beta \, \E \Bigl[\sup_{t\in [0,T]}\int_{\mathcal O} \abs{u^n}^2\,dx\Bigl]
		\\ &+S_4
		\E\Bigl[ \int_0^T\int_{\mathcal O} \abs{u^n}^2 \,dx\,dt
		+T \abs{\mathcal O}\Bigl]
		\\ &  \leq S_5+  \beta\, \E\left[ \sup_{t\in [0,T]} 
		\norm{u^n(t)}_{L^2(\mathcal O)}^2\right].
	\end{split}
\end{equation}
Herein, the constants $S_1,\dots,S_5>0$ 
are independent of $n$. Finally, we choose $\beta>0$ small enough such that we establish \eqref{apriori-sup}.
This concludes the proof of Lemma \ref{lémmaaprioriés}.
}
\end{proof}

The following corollary, involving higher–moment estimates, is an immediate consequence of Lemma \ref{lémmaaprioriés} together with the BDG inequality. 

\begin{cor}\label{cor:Lq0-est}
Along with the assumptions stated in Lemma \ref{lémmaaprioriés}, 
let us further assume that $q_0>9/2$. 
Then there exists a constant $C > 0$, which remains 
independent of $n$, satisfying 
the following estimates:
\begin{equation}\label{eq:Lq0-est}
	\begin{split}
		&\E\left[ 
		\, \norm{u^n(t)}_{L^\infty(0,T;L^2(\mathcal{O}))}^{q_0}\right]+\E\left[ \,
	\norm{\nabla A(u^n)}_{L^2((0,T)\times \mathcal{O})}^{q_0}
	\right] \le C.
	\end{split}
\end{equation}

\end{cor}

\begin{proof}
{
In view of \eqref{eq:v+w-tmp2}, we have the following 
estimate for any $(\omega,t)\in D \times [0,T]$: 
\begin{align*}
	\sup_{0\le \tau\le t}\norm{u^n(\tau)}_{L^2(\Omega)}^2 
	 &\leq
	\norm{u^n(0)}_{L^2(\Omega)}^2 
	+ C+ C\int_0^t \norm{u^n(s)}_{L^2(\Omega)}^2 \,ds
	\\ & \quad \qquad\quad
	+C\sup_{0\le \tau\le t}\abs{\int_0^{\tau}\int_{\Omega} 
	u^n \sigma_n(u^n) \,dx \,dW^{n}(s)},
\end{align*}
for some constant $C>0$ independent of $n$. Next, we elevate both sides of the inequality to the power 
 $q_0/2$, take expectations, and apply a series of elementary inequalities to obtain 
\begin{equation}\label{eq:pmoment}
	\begin{split}
		\E\left[ \,\sup_{0\le \tau\le t}
		\norm{u^n(\tau)}_{L^2(\Omega)}^{q_0}\right]  &\le 
		C\left (1+\E\left[ \norm{u^n(0)}_{L^2(\Omega)}^{q_0}\right] \right)
		+C\int_0^\tau \E\left[ 
		\norm{u^n(s)}_{L^2(\Omega)}^{q_0}\right] \,ds
		+ E_\sigma,
	\end{split}
\end{equation}
where
\begin{align*}
	E_\sigma & := \E\left[\, \sup_{0\le \tau\le t}\abs{\int_0^{\tau}
	\int_{\Omega} u^n \sigma^n(u^n) \,dx \,dW^{n}(s)}^{\frac{q_0}{2}}\,\right].
\end{align*}
Working as in \eqref{Esup1} and using a martingale inequality, we obtain
\begin{equation}\label{Esup-vq0}
	\begin{split}
		E_\sigma & \le C \E \left[ \left(\int_0^t \sum_{k=1}^n 
		\abs{\int_{\Omega} u^n \sigma^n_k(u^n) \,dx}^2 \,ds\right)^{\frac{q_0}{4}} \right]
		\\ & \le C \E \left[\left( \int_0^t  
		\left(\int_{\Omega} \abs{u^n}^2\,dx\right) 
		\left(\sum_{k=1}^n\int_{\Omega} \abs{\sigma^n_k(u^n)}^2 \,dx 
		\right)\,ds\right)^{\frac{q_0}{4}} \right] \\ & \le C \E \left[
		\left(\sup_{\tau\in [0,t]}\int_{\Omega} \abs{u^n}^2\,dx\right)^{\frac{q_0}{4}}
		\left( \int_0^t\sum_{k=1}^n\int_{\Omega} 
		\abs{\sigma^n_k(u^n)}^2 \,dx\,ds \right)^{\frac{q_0}{4}}\right]
		\\ & \le \beta
		\E \left[\left(\sup_{\tau\in [0,t]}\int_{\Omega} 
		\abs{u^n}^2\,dx\right)^{\frac{q_0}{2}}\right]
		\\ &+C(\beta)\E\left[\left( \int_0^t\sum_{k=1}^n\int_{\Omega} 
		\abs{\sigma^n_k(u^n)}^2 \,dx\,ds \right)^{\frac{q_0}{2}}\right]
		\\ & \le \beta \,
		\E \left[ \sup_{\tau\in [0,t]} \norm{u^n(\tau)}_{L^2(\Omega)}^{q_0}\right]
		+C \left(1+\int_0^t \E \left[\norm{u^n(s)}_{L^2(\Omega)}^{q_0} \,ds\right]\right),
	\end{split}
\end{equation}
for any $\delta>0$. With $\beta$ sufficiently small, combining \eqref{Esup-vq0} with \eqref{eq:pmoment} yields
\begin{equation}\label{eq:pmoment2}
	\begin{split}
		\E\left[ \,\sup_{0\le \tau\le t}\norm{u^n(\tau)}_{L^2(\Omega)}^{q_0} \, \right] 
		& \le C\E\left[ \norm{u^n(0)}_{L^2(\Omega)}^{q_0}\right] 
		 +C \left(1+\int_0^t \E \left[\norm{u^n(s)}_{L^2(\Omega)}^{q_0} \,ds\right]\right),
	\end{split}
\end{equation}
for some constants $C,C'>0$ independent of $n$. 
Now an application of Gr\"onwall's inequality 
yields the first desired result \eqref{eq:Lq0-est}.

Finally, proceeding as in the proof of the first part of \eqref{eq:Lq0-est} and making use of \eqref{firstpart}, \eqref{Esup-vq0} for $A(u^n)$, and \eqref{apriori-sup}, we conclude that
$$
\E\left[ \abs{\int_0^t \int_{\Omega}  \abs{\nabla A(u^n)}^2 \,dx \,ds}^{\frac{q_0}{2}}\right] 
\le C,
$$
and \eqref{eq:Lq0-est} follows.
}
\end{proof}
Setting the probabilistic variable aside, in order to establish the strong convergence of the sequence of Faedo-Galerkin solutions, we need to establish, in addition to the previous a priori estimate, a so-called temporal translation estimate, which is stated in the upcoming lemma.
\begin{lem}\label{lémmatranslaté}
	Let Assumptions \ref{a1}-\ref{a5} be satisfied and assume further that $0 \leq u_0 \leq \overline{u},\; \mathbb{P}\text{-a.s},$ where $\overline{u}$ is as defined in Assumption \ref{a1} and $\sigma(0)=\sigma(\overline{u})=0$. For $\varepsilon_n>0$, extend the sequence of Faedo-Galerkin solutins $\left(u^n(t)\right)_{0 \leq t \leq T}$ by $0$ outside $[0,T]$. Then, there exits $C_{20}>0$ such that 
	\begin{equation}\label{tmprl}
		\mathbb{E} \underset{\vert \tau \vert \in (0,\delta)}{\sup} \Vert u^n(t+\tau) -u^n(t)\Vert_{H^1(\mathcal{O})^*}  \leq C_{20} \delta, \; \forall t \in [0,T],
	\end{equation}
	for any sufficiently small $\delta>0$.
\end{lem} 
\begin{proof}
	Let $\tau \in (0,\delta)$ for $\delta>0$. Then, it holds that 
	\begin{align*}
		\Vert u^n(t+\tau)-u^n(t) \Vert_{H^1(\mathcal{O})^*}&=\underset{\Vert \psi \Vert_{H^1(\mathcal{O})} \leq 1}{\sup} \left\{\bigg \vert \int_{\mathcal{O}} (u^n(t+\tau)-u^n(t)) \psi dx  \bigg \vert \right\}\\
		&\leq \sum_{i=1}^3 \mathscr{I}_i(\tau),
	\end{align*}
	where, by H\"{o}lder and Cauchy-Schwarz inequalities and Lemmas \ref{lemmaa} and \ref{lémmaaprioriés}
{\color{black}
	\begin{equation}\label{est-dual}\begin{split}
		\mathscr{I}_1(\tau)&:=\bigg \Vert \int_{t}^{t + \tau}  \Pi_n\left( \div\left(a_{\varepsilon_n}(u^n(s)) \nabla u^n(s)-u^n(s)\nabla \mathcal{K} \ast u^n(s)\right)\right)ds \bigg \Vert _{H^1(\mathcal{O})^*} \\
		&:=\underset{\Vert \psi \Vert_{H^1(\mathcal{O})} \leq 1}{\sup}\left\{\bigg \vert \int_{t}^{t + \tau} \int_{\mathcal{O}} \left(a_{\varepsilon_n}(u^n(s)) \nabla u^n(s)-u^n(s)\nabla \mathcal{K} \ast u^n(s)\right) \cdot \nabla \Pi_n(\psi) \bigg \vert dxds \right\}\\
		&\leq \underset{\Vert \psi \Vert_{H^1(\mathcal{O})} \leq 1}{\sup} \Biggl\{ C_{21} \sqrt{\tau} \left(\Vert \nabla A(u^n)\Vert_{L^2(0,T;L^2(\mathcal{O})^d)}+\varepsilon_n\Vert \nabla u^n\Vert_{L^2(0,T;L^2(\mathcal{O})^d)}+\Vert  u^n\nabla \mathcal{K} \ast u^n\Vert_{L^2(0,T;L^2(\mathcal{O}))}\right)\\
&\qquad \qquad    \qquad \times \Vert \nabla \Pi_n(\psi)\Vert_{L^2(\mathcal{O})^d} \Biggl\} \\
&\leq \underset{\Vert \psi \Vert_{H^1(\mathcal{O})} \leq 1}{\sup} \Biggl\{ C_{21} \sqrt{\tau} \left(\Vert \nabla A(u^n)\Vert_{L^2(0,T;L^2(\mathcal{O})^d)}+\varepsilon_n\Vert \nabla u^n\Vert_{L^2(0,T;L^2(\mathcal{O})^d)}+\Vert  u^n\Vert_{L^2(0,T;L^2(\mathcal{O}))}\right)\\
&\qquad \qquad    \qquad \times \Vert \nabla \Pi_n(\psi)\Vert_{L^2(\mathcal{O})^d} \Biggl\} \\
		&\leq C_{22} \sqrt{\tau}.
\end{split}	\end{equation}
}
Similarly,
	\begin{align}
	\mathscr{I}_2(\tau)&:=\bigg \Vert \int_{t}^{t + \tau}  \Pi_n\left( f_M(u^n(s))\right)ds \bigg \Vert _{H^1(\mathcal{O})^*}:=\underset{\Vert \psi \Vert_{H^1(\mathcal{O})} \leq 1}{\sup}\left\{\bigg \vert \int_{t}^{t+\tau} \int_{\mathcal{O}} f_M(u^n(s)) \Pi_n(\psi) \bigg \vert dxds\right\} \leq C_{23} \sqrt{\tau}.
\end{align}
Thus, 
\begin{equation}\label{partaa}
	\mathbb{E} \underset{\vert \tau \vert \in (0,\delta)}{\sup}  \mathscr{I}_1(\tau)+\mathscr{I}(\tau)  \leq (C_{22}+C_{23}) \sqrt{\tau}.
\end{equation}
Now, by Cauchy-Shwarz and BDG inequalities (Remark \ref{rmbdg}) along with Assumption \ref{a5} and  Lemma \ref{lemmaa}, we acquire
\begin{align}\label{partbb}
\mathbb{E} \underset{\vert \tau \vert \in (0,\delta)}{\sup} 	\mathscr{I}_3(\tau)&:= \mathbb{E} \underset{\vert \tau \vert \in (0,\delta)}{\sup} \bigg \Vert \sum_{k=1}^{n} \int_{t}^{t + \tau}  \sigma_{k}^n(u^n(s)) d\mathcal{W}_s \bigg \Vert _{H^1(\mathcal{O})^*}\nonumber \\
	&:=\mathbb{E} \underset{\vert \tau \vert \in (0,\delta)}{\sup}  \underset{\Vert \psi \Vert_{H^1(\mathcal{O})} \leq 1}{\sup} \left\{\int_t^{t+\tau} \int_{\mathcal{O}}  \sigma_{k}^n(u^n(s)) \psi d\mathcal{W}_s dx\right\} \nonumber\\
	&\leq \mathbb{E} \underset{\vert \tau \vert \in (0,\delta)}{\sup}  \underset{\Vert \psi \Vert_{H^1(\mathcal{O})} \leq 1}{\sup} \left\{\bigg \Vert    \int_t^{t+\tau}  \sigma_{k}^n(u^n(s)) d\mathcal{W}^n_s \bigg \Vert_{L^2(\mathcal{O})} \Vert \psi \Vert_{L^2(\mathcal{O})} \right\} \nonumber\\
	&\leq \mathbb{E} \underset{\vert \tau \vert \in (0,\delta)}{\sup}  \left\{\bigg \Vert    \int_t^{t+\tau}  \sigma_{k}^n(u^n(s)) d\mathcal{W}^n_s \bigg \Vert_{L^2(\mathcal{O})} \right\} \nonumber\\
	&\leq \mathbb{E} \left(\sum_{k=1}^n  \int_{t}^{t+\tau} \int_{\mathcal{O}} \left(\sigma_{k}^n(u^n(s))\right)^2 dx ds\right)^{\frac{1}{2}} \nonumber\\
	&\leq \sqrt{\tau} C_{1} \left(1+\mathbb{E} \Vert u^n \Vert_{L^\infty(0,T;L^2(\mathcal{O}))}\right)  \nonumber\\
	&\leq C_{24} \sqrt{\tau}.
\end{align} 
The result follows by combining Inequalities \eqref{partaa} and \eqref{partbb}.
\end{proof}
\begin{rem}
As a direct consequence of Lemma \ref{lémmatranslaté} along with the monotonicity of $A$, one can prove that $\left(A(u^n(t))\right)_{0 \leq t \leq T}$ satisfies Estimate \eqref{tmprl}. 
\end{rem}
\section{Retrieval of unique weak martingale solutions in the degenerate case}\label{section7}

\subsection{Tightness and Skorokhod almost-sure representations}\label{section5}
Unlike the deterministic case, the a priori estimates given by Lemmas  \ref{lémmaaprioriés}  and \ref{lémmatranslaté} are not sufficient to ensure the convergence of the constructed sequence of Faedo-Galerkin solutions to the desired martingale solution by compactness arguments. Indeed, we additionally need to obtain the convergence (in a suitable sense) with respect to the probabilistic variable. As it turns out, when dealing with sequences of probability measures, the notion of compactness is replaced with tightness, whose definition is recalled below.
\begin{defi}(\cite[p. 32]{da2014stochastic})
	Let $\mathscr{X}$ be a Polish space and $\mathcal{B}(\mathscr{X})$ its corresponding collection of Borel subsets. Then, a sequence of probability measures $\left(\mu^n \right)_{n \in \mathbb{N}^*} \subset \left(\mathscr{X},\mathcal{B}(\mathscr{X})\right)$ is said to be tight if the following condition is satisfied: 
	$$
	\forall \varepsilon>0, \quad \exists \mathscr{C}_\varepsilon \subset \mathscr{X}, \quad \mu^n\left(\mathscr{C}_\varepsilon\right)>1-\varepsilon,\; \forall n \in \mathbb{N}^*,
	$$
	where $\mathscr{C}_\varepsilon$ is a compact set.
\end{defi}
\noindent
With the above definition in mind, we introduce the following useful result due to Prokhorov:
\begin{lem}(\cite[Theorem 2.3]{da2014stochastic})
	\label{prokhorow}
	A sequence probability measures $\left(\mu^n\right)_{n \in \mathbb{N}^*} \subset \left(\mathscr{X},\mathcal{B}(\mathscr{X})\right)$ is tight if and only if there exists a (non-relabeled) sequence $\left(\mu^n\right)_{n \in \mathbb{N}^*} \subset \left(\mathscr{X},\mathcal{B}(\mathscr{X})\right)$ which converges weakly to a probability measure $\mu$. That is, 
	$$
	\int_{\mathscr{X}} \phi(\omega) d\mu^n \rightarrow \int_{\mathscr{X}} \phi(\omega) d\mu,
	$$
	where $\phi : \mathscr{X} \rightarrow \mathbb{R}$ is a continuous bounded function. 
\end{lem}

\noindent
We will make use of the previous theorem as follows. First, we recall that any random variable $X : \Omega \rightarrow \mathbb{R}$ induces a probability measure $\mathcal{L}$ on $\left(\mathscr{X},\mathcal{B}(\mathscr{X})\right)$ given by 
$\mathcal{L}:=\mathbb{P} \circ X^{-1}$. Then, for a given sequence $\left(X_n\right)_{n \in \mathbb{N}^*}$ with a corresponding induced sequence of  probability measures $\left(\mathscr{L}_n\right)_{n \in \mathbb{N}^*}$, we will have to unify the subsequences for which the $\mathbb{P}$\text{-a.s} convergence of $\left(X_n\right)_{n \in \mathbb{N}^*}$ and the weak convergence of $\left(\mathscr{L}_n\right)_{n \in \mathbb{N}^*}$ are ensured. We achieve this by relying on Skorokhod's representation theorem, which we recall below.
\begin{lem}(\cite[Theorem 2.4]{da2014stochastic})\label{skorokhod}
	Let $\left(X_n\right)_{n \in \mathbb{N}^*}$ be a sequence of random variables with a corresponding sequence of probability measures $\left(\mathscr{L}_n\right)_{n \in \mathbb{N}^*}$ converging weakly to a probabilistic measure $\mathcal{L}$. Then, $\left(X_n,\mathscr{L}_n\right)_{n \in \mathbb{N}^*}$ has a Skorokhod's representation. That is, there exists a probability space $\left(\tilde{\Omega},\tilde{\mathcal{F}},\tilde{\mathbb{P}}\right)$, a sequence of random variables $\left(\tilde{X_n}\right)_{n \in \mathbb{N}^*}$, a random variable $X$, both defined in $\tilde{\Omega}$ such that the following assertions are satisfied: 
	\begin{enumerate}
		\item For all $n \in \mathbb{N}^*$, the law of $\tilde{X_n}$ is $\mathscr{L}_n$; 
		\item  The law of $\tilde{X}$ is $\mathcal{L}$;  
		\item $\tilde{X_n} \rightarrow \tilde{X}$ $\mathbb{P}$\text{-a.s} as $n \rightarrow +\infty$.
	\end{enumerate}
\end{lem}
\noindent
Prior to applying the preceding lemmas, we first construct the suitable Polish space that we will be working on, which in our case, is defined as follows:  
$$
\mathscr{X}:= \mathscr{X}_1 \times \mathscr{X}_2 \times \mathscr{X}_0,
$$
where
$$
\mathscr{X}_1:=L^2(0,T;L^2(\mathcal{O})) \bigcap C([0,T];H^1(\mathcal{O})^*),
$$
$$
\mathscr{X}_2:=C([0,T];\mathbb{H}_0),
$$
and
$$
\mathscr{X}_0:=L^2(\mathcal{O}).
$$
We further denote by $\mathcal{B}(\mathscr{X})$ the associated $\sigma$-algebra of Borel subsets of $\mathscr{X}$. Now, we define the following measurable mapping: 
\begin{align*}
	\mathscr{J}_n :\; &(\Omega,\mathcal{F},\mathbb{P}) \rightarrow (\mathscr{X},\mathcal{B}(\mathscr{X}))\\
	&\omega \mapsto \left(u^n(\omega),\mathcal{W}^n(\omega),u_0^n(\omega)\right).
\end{align*}
To conclude our theoretical setting, we further consider the following sequence of probability measures: 
\begin{equation}\label{equameas}
	\mathscr{L}_n(A):=\left(\mathbb{P}\circ \mathscr{J}_n^{-1}\right)(A):=\mathbb{P}\left( \mathscr{J}_n^{-1}(A)\right), \quad \forall A \in \mathcal{B}(\mathscr{X}),
\end{equation}
such that
$$
\mathscr{L}_n=\mathscr{L}_{u^n} \times \mathscr{L}_{\mathcal{W}^n} \times \mathscr{L}_{u^n_0},
$$
where $\mathscr{L}_{u^n}$, $\mathscr{L}_{\mathcal{W}^n}$ and $\mathscr{L}_{u^n_0}$ are the laws of $u^n$, $\mathcal{W}^n$ and $u^n_0$, defined on $\left(\mathscr{X}_1,\mathcal{B}(\mathscr{X}_1)\right)$, $\left(\mathscr{X}_2,\mathcal{B}(\mathscr{X}_2)\right)$ and $\left(\mathscr{X}_0,\mathcal{B}(\mathscr{X}_0)\right)$, respectively.

With the previous setting, in order to achieve the tightness of the sequence of probability measures $\left(\mathscr{L}_n\right)_{n \in \mathbb{N}^*}$ given by Formula \eqref{equameas}, due to the presence of degenerate diffusion, in contrast to earlier contributions we will first need to separately establish the tightness of the sequence of probabilistic laws ${\left(\mathcal{L}_{A(u_{n})}\right)}_{n \in \mathbb{N}^*}$ (corresponding to $\left(A(u^n)\right)_{n \in \mathbb{N}^*}$). We highlight that in the non-degenerate case, one does not need to establish such a result.

\begin{lem}\label{tightaun}
	Let Assumptions \ref{a1}-\ref{a5} be satisfied and assume further that $0 \leq u_0 \leq \overline{u},\; \mathbb{P}\text{-a.s},$ where $\overline{u}$ is as defined in Assumption \ref{a1} and $\sigma(0)=\sigma(\overline{u})=0$. Then, the sequence of probability measures ${\left(\mathcal{L}_{A(u_{n})}\right)}_{n \in \mathbb{N}^*}$ is uniformly tight.
\end{lem}
\begin{proof}
	For every $\mathscr{R}_1>0$, by Chebyshev's inequality and taking into account Lemma \ref{lémmaaprioriés}, it holds that
	
	\begin{equation} \label{cheb1}
		\mathbb{P}\left(\left\{\omega \in \Omega :  \quad \Vert  u^n(\omega) \Vert_{L^\infty(0,T;L^2(\mathcal{O}))} > \mathscr{R}_1 \right\}\right) \leq \dfrac{1}{\mathscr{R}_1} \mathbb{E} \Vert  u^n(\omega) \Vert_{L^\infty(0,T;L^2(\mathcal{O}))} \leq \dfrac{C_{25}}{\mathscr{R}_{1}}.
	\end{equation}
	On the other hand, for every $\eta>0$ and $\mathscr{R}_2>0$, by using once again Chebyshev's inequality and from Lemma \ref{lémmatranslaté}, we acquire
	\begin{equation} \label{cheb2}
		\mathbb{P}\left(\left\{\omega \in \Omega :  \quad \underset{\tau \in (0,\eta)}{\sup} \Vert u(.+\tau)-u(.) \Vert_{L^\infty(0,T-\tau;H^1(\mathcal{O})^*)} > \mathscr{R}_2 \right\}\right)\leq \dfrac{C_{26}}{\mathscr{R}_{2}}.
	\end{equation}
	Now, we consider two sequences $\left(\nu^n\right)_{n \in \mathbb{N}^*} \subset  (0, +\infty)$ and $\left(\eta_n\right)_{n \in \mathbb{N}^*} \subset (0, +\infty)$ such that the following assertions are fulfilled: 
	\begin{enumerate}
		\item $\left(\nu^n,\eta_n\right) \rightarrow \left(0,0\right)$ as $n \rightarrow + \infty$;
		\item $\displaystyle \sum_{n \in \mathbb{N}^*} \dfrac{\eta_n^{\frac{1}{4}}}{\nu^n}<\infty.$ 
	\end{enumerate}
	\noindent
	Additionally, we consider the following normed space: 
	$$
	\mathscr{Z}_{\eta_n,\nu^n}:=\left\{\Psi \in L^\infty(0,T;L^2(\mathcal{O))} \bigcap L^2(0,T;H^1(\mathcal{O})):\; \underset{n \in \mathbb{N}^*}{\sup} \dfrac{1}{\nu^n} \underset{\tau \in (0,\eta_n)}{\sup} \Vert \Psi(.+\tau)-\Psi(.) \Vert_{L^\infty(0,T-\tau;H^1(\mathcal{O})^*)}<\infty\right\},
	$$
	equipped with the norm: 
	$$
	\Psi \mapsto \Vert \Psi \Vert_{\mathscr{Z}_{\eta_n,\nu^n}}:= \Vert \Psi \Vert_{L^\infty(0,T:L^2(\mathcal{O}))}+ \Vert \Psi \Vert_{L^2(0,T:H^1(\mathcal{O}))}+\underset{n \in \mathbb{N}^*}{\sup} \dfrac{1}{\nu^n} \underset{\tau \in (0,\eta_n)}{\sup} \Vert \Psi(.+\tau)-\Psi(.) \Vert_{L^\infty(0,T-\tau;H^1(\mathcal{O})^*)}.
	$$
	By Aubin-Lions compactness lemma \cite{simon1986compact}, it holds that the embedding
	\begin{equation}\label{émbéd}
		\mathscr{Z}_{\eta_n,\nu^n} \hookrightarrow L^2(0,T;L^2(\mathcal{O})) \bigcap C([0,T];H^1(\mathcal{O})^*)
	\end{equation} 
	is compact.

	\noindent
	Let $\delta>0$. Keeping Aubin-Lions compactness lemma in mind, we consider the following compact subset of $L^2(0,T;L^2(\mathcal{O})))$: 
	$$
	\mathscr{C}_1^\delta:=\left\{\Psi \in \mathscr{Z}_{\eta_n,\nu^n} : \quad  \Vert \Psi \Vert_{\mathscr{Z}_{\eta_n,\nu^n}} \leq \mathscr{R}_1^\delta \right\},
	$$
	where $\mathscr{R}_1^\delta>0$ will be chosen accordingly.  
	
	\noindent
	$\forall n \in \mathbb{N}^*$, it holds that 
	\begin{align*}
		&\mathbb{P}\left(\left\{\omega \in \Omega :  \quad   A(u^n)(\omega) \notin \mathscr{C}_1^\delta \right\}\right)\\
		&\leq \dfrac{1}{3}  \mathbb{P}\left(\left\{\omega \in \Omega :  \quad \Vert  A(u^n)(\omega) \Vert_{L^\infty(0,T;L^2(\mathcal{O}))} > \mathscr{R}_1^\delta  \right\}\right)\\
		&+ \dfrac{1}{3} \mathbb{P}\left(\left\{\omega \in \Omega :  \quad \Vert  A(u^n)(\omega) \Vert_{L^2(0,T;H^1(\mathcal{O}))} > \mathscr{R}_1^\delta  \right\}\right)\\
		&+ \dfrac{1}{3} \mathbb{P}\left(\left\{\omega \in \Omega :  \quad \underset{\tau \in (0,\eta_n)}{\sup} \Vert A(u^n)(.+\tau)(\omega)-A(u^n)(.)(\omega) \Vert_{L^\infty(0,T-\tau;H^1(\mathcal{O})^*)} > \mathscr{R}_1^\delta \nu^n  \right\}\right)\\
		&=:\mathcal{P}_1+\mathcal{P}_2+\mathcal{P}_3.
	\end{align*}
	By the definition of $A$ (Formula \eqref{défiA}) and taking Estimates \eqref{cheb1}, \eqref{cheb2}  and  Lemma \ref{lémmaaprioriés} into account, we directly obtain by Chebyshev's inequality
	\begin{equation*}
		\mathcal{P}_1 \leq \dfrac{C_{27}}{\mathscr{R}_{1}^\delta}, 
	\end{equation*}
	\begin{equation*}
		\mathcal{P}_2 \leq \dfrac{C_{28}}{\mathscr{R}_{1}^\delta}, 
	\end{equation*}
	and
	\begin{align}
		\mathcal{P}_3 &\leq \dfrac{1}{\mathscr{R}_1^\delta} \displaystyle \sum_{n \in \mathbb{N}^*}  \dfrac{1}{\nu^n} \mathbb{E} \underset{\tau \in (0,\eta_n)}{\sup} \Vert A(u^n)(.+\tau)(\omega)-A(u^n)(.)(\omega) \Vert_{L^\infty(0,T-\tau;H^1(\mathcal{O})^*)} \nonumber\leq \dfrac{C_{29}}{\mathscr{R}_{1}^\delta} \sum_{n \in \mathbb{N}^*} \dfrac{\eta_n^{\frac{1}{4}}}{\nu^n}\leq \dfrac{C_{30}}{\mathscr{R}_{1}^\delta}.
	\end{align}
	Consequently, we can adequately choose $\mathscr{R}_{1}^\delta$ such that
	\begin{equation}\label{tight1}
		\mathscr{L}_{A(u^n)}\left(\mathscr{X}_1 \setminus \mathscr{C}_1^\delta\right):=\mathbb{P}\left(\left\{\omega \in \Omega :  \quad \Vert  A(u^n)(\omega) \Vert_{L^\infty(0,T;L^2(\mathcal{O}))}\right\}\right) \leq \dfrac{\delta}{6}.
	\end{equation}
	This concludes the proof.
\end{proof}
\color{black}
\noindent
The result of tightness of the sequence of probability measures $\left(\mathscr{L}_n\right)_{n \in \mathbb{N}^*}$ immediately follows. 
\begin{lem}\label{tight}
	Let Assumptions \ref{a1}-\ref{a5} be satisfied and assume further that $0 \leq u_0 \leq \overline{u},\; \mathbb{P}\text{-a.s},$ where $\overline{u}$ is as defined in Assumption \ref{a1} and $\sigma(0)=\sigma(\overline{u})=0$. Then, the sequence of probability measures $\left(\mathscr{L}_n\right)_{n \in \mathbb{N}^*}$ given by Formula \eqref{equameas} is uniformly tight.
\end{lem}
\begin{proof}
	
	\noindent
	Let $\phi : \mathscr{X} \rightarrow \mathbb{R}$ be a continuous bounded function.  In order to address the tightness of $\left(\mathscr{L}_{u^n}\right)_{n \in \mathbb{N}^*}$, we use the tightness of $\left(\mathscr{L}_{A(u^n)}\right)_{n \in \mathbb{N}^*}$. Indeed by Prokhorov's theorem, there exists a probabilistic law $\tilde{\mathscr{L}}$ corresponding to a random variable $v$ such that (up to a subsequence) 
	$$
	\int_{\mathscr{X}} \phi(\omega) d\mathscr{L}_{A(u^n)} \rightarrow \int_{\mathscr{X}} \phi(\omega) d\tilde{\mathscr{L}}.
	$$
	
	\noindent
	Now, taking into account that $A^{-1}$ is well defined and continuous along with the fact that $(A(u^n))_{n \in \mathbb{N}^*}$ converges in distribution to $v$, we deduce that 
	$$
	\int_{\mathscr{X}} \phi(\omega) d\mathscr{L}_{u^n} \rightarrow \int_{\mathscr{X}} \phi(\omega) d{\mathscr{L}_{A^{-1}(v)}}.
	$$
	\noindent
	Hence, by using Prokhorov's theorem, we deduce the tightness of $\left(\mathscr{L}_{u^n}\right)_{n \in \mathbb{N}^*}$.
	
	\noindent 
	To address the tightness of $\left(\mathcal{L}_{{\mathcal{W}}^n}\right)_{n \in \mathbb{N}^*}$ and $\left(\mathcal{L}_{{u_0}^n}\right)_{n \in \mathbb{N}^*}$, we recall that the finite series $\mathcal{W}^n$ is $\mathbb{P}$\text{-a.s} convergent in $C([0,T];\mathbb{H}_0)$ as $n\rightarrow +\infty$, which in turns yields the weak convergence of $\mathscr{L}_{\mathcal{W}^n}$. Therefore, by using Prokhorov's theorem, we deduce the tightness of $\mathscr{L}_{\mathcal{W}^n}$ (up to a subsequence). The result immediately follows by recalling the definition of $\left(\mathscr{L}_n\right)_{n \in \mathbb{N}^*}$.
\end{proof}
As a consequence, the following second intermediate result of this section reads as follows.
\begin{lem}
	Let Assumptions \ref{a1}-\ref{a5} be satisfied and assume further that $0 \leq u_0 \leq \overline{u},\; \mathbb{P}\text{-a.s},$ where $\overline{u}$ is as defined in Assumption \ref{a1} and $\sigma(0)=\sigma(\overline{u})=0$. Then, (up to a subsequence)  $\left(u^n,\mathscr{L}_n\right)_{n \in \mathbb{N}^*}$ has a Skorokhod's representation.
\end{lem}
\begin{proof}
	By Lemma \ref{tight}, there exists a (non-relabeled) subsequence $\left(\mathscr{L}_n\right)_{n \in \mathbb{N}^*}$ converging weakly to a limit which we denote by $\mathscr{L}$. Now, let $n \in \mathbb{N}^*$. By using Lemma \ref{skorokhod}, we obtain the existence of a new probabilistic space $\left(\tilde{\Omega},\tilde{\mathcal{F}},\tilde{\mathbb{P}}\right)$, random variables $\left(\tilde{u}^n,\tilde{\mathcal{W}}_n,\tilde{u}_0^n\right)$ and  $\left(\tilde{u},\tilde{\mathcal{W}},\tilde{u}_0\right)$ with respective laws $\tilde{\mathscr{L}}_n$ and $\tilde{\mathscr{L}}$ such that the law $\left(\tilde{u}^n,\tilde{\mathcal{W}}_n,\tilde{u}_0^n\right)$ is $\mathscr{L}_n$ and the law of $\left(\tilde{u},\tilde{\mathcal{W}},\tilde{u}_0\right)$ is $\mathscr{L}$.
%
\end{proof}
\begin{rem}
	Note that by the equality of laws $\mathscr{L}_n=\tilde{\mathscr{L}}_n$, it holds that the new sequence $\left(\tilde{u}^n\right)_{n \in \mathbb{N}^*}$ satisfies the uniform a priori estimates given by Lemmas \ref{lemmaa} and \ref{lémmaaprioriés}.
\end{rem}
\subsection{Existence}
Let us now proceed to the last step towards proving our main result. To this end, we first consider the following stochastic basis: 
$$
\tilde{\mathscr{S}}:=\left(\tilde{\Omega},\tilde{\mathcal{F}},\{\tilde{\mathcal{F}}_t\}_{0 \leq t \leq T},\tilde{\mathbb{P}},\tilde{\mathcal{W}}\right),
$$
where
$$
\tilde{\mathcal{F}}_t:={\Sigma}\left(\Sigma\left(\left(\tilde{u},\tilde{\mathcal{W}},\tilde{u}_0\right)|_{(0,T)}\right) \bigcup \left\{N \in \tilde{F}: \quad \tilde{\mathbb{P}}(N)=0\right\} \right),
$$
such that $.|_{(0,T)}$ stands for the restriction, of a given stochastic process, to the interval $(0,T)$, and $\Sigma(\cdot)$ denotes the $\sigma$-algebra generated by a given set.
\noindent
Additionally, $\tilde{W}$ is a cylindrical Wiener process defined as the $\mathbb{P}$-almost-sure limit of the finite cylindrical Wiener process $\tilde{\mathcal{W}^n}$ (see Convergence $\eqref{convsko}_{3}$). Otherwise written
$$
\tilde{\mathcal{W}}_t:= \displaystyle \sum_{k \in \mathbb{N}^*} \tilde{W}_k(t) \varphi_k, \quad \forall t \in [0,T],$$
where $\left\{\left(\tilde{W}_k(t)\right)_{0 \leq t \leq T}, k \in \mathbb{N}\right\}$ is a sequence of independent $\mathcal{F}_t$-adapted one-dimensional Wiener processes and $\left(\varphi_k\right)_{k \in \mathbb{N}^*}$ is as introduced in Section \ref{section2}. 

\noindent
Thus, to conclude the main result, it remains to prove that Identity \eqref{idn} is satisfied by $\left(\tilde{u},\tilde{W}\right)$ uniquely for $\tilde{u}$. This is the subject of the following lemma:
\begin{lem}
	Let Assumptions \ref{a1}-\ref{a5} be satisfied and assume further that $0 \leq u_0 \leq \overline{u},\; \mathbb{P}\text{-a.s},$ where $\overline{u}$ is as defined in Assumption \ref{a1} and $\sigma(0)=\sigma(\overline{u})=0$. Assume further that the probability measure $\mu$ satisfies 
\begin{equation}\label{condinit}
\exists q>2, \quad \int_{L^2(\mathcal{O})} \Vert u \Vert_{L^2(\mathcal{O})}^q d\mu(u)< \infty.
\end{equation}
 Then, it holds that $\left(\tilde{u},\tilde{W},\tilde{u}_0\right)$ satisfies $\tilde{\mathbb{P}}\text{-a.s}$  the following identity: 
\begin{equation}\label{idntild}
	\begin{split}
		\int_\mathcal{O} \tilde{u}(t)v\; dx&= \int_\mathcal{O} \tilde{u}_0v\; dx- \int_{0}^t \int_{\mathcal{O}} a(\tilde{u}(t))\nabla \tilde{u}(t) \cdot \nabla v\; dxdt+\int_{0}^t \int_{\mathcal{O}} \tilde{u}(t)\left(\nabla \mathcal{K} \ast \tilde{u}(t)\right) \cdot \nabla v\; dxdt\\
		&+\int_{0}^t \int_{\mathcal{O}} f(\tilde{u}(t))v\; dxdt+\int_{0}^t \int_{\mathcal{O}} \sigma(\tilde{u}(t))v\; dx d\tilde{\mathcal{W}}_s, \quad \forall v \in H^1(\mathcal{O}), \; \forall t \in [0,T],
	\end{split}
\end{equation}
where $\tilde{u}_0$ has a probabilisic law $\mu$ and $0 \leq \tilde{u}_0 \leq \overline{u}\; \mathbb{P}\text{-a.s}$ and $\sigma(0)=\sigma(\overline{u})=0$.
\end{lem}
\begin{proof}
We recall that the sequence of Faedo-Galerkin solutions in the non-degenerate case for $\varepsilon_n:=\dfrac{1}{n}, \; n \in \mathbb{N}^*$. The latter satisfies $\tilde{\mathbb{P}}\text{-a.s},$ $\forall t \in [0,T]$ the following identity:
\begin{equation}\label{suitéidntild}
	\begin{split}
		\int_\mathcal{O} \tilde{u}^n(t)v\; dx&= \int_\mathcal{O} \tilde{u}_0^n v\; dx- \int_{0}^t \int_{\mathcal{O}} a_{\varepsilon_n}(\tilde{u}^n(t))\nabla \tilde{u}^n(t) \cdot \nabla \Pi_n v\; dxdt+\int_{0}^t \int_{\mathcal{O}} \tilde{u}^n(t)\left(\nabla \mathcal{K} \ast \tilde{u}^n(t)\right) \cdot \nabla \Pi_n v\; dxdt\\
		&+\int_{0}^t \int_{\mathcal{O}} f_M(\tilde{u}^n(t)) \Pi_n v\; dxdt+\int_{0}^t \int_{\mathcal{O}} \sigma^n(\tilde{u}^n(t)) \Pi_nv\; dx d\tilde{\mathcal{W}}^n_s, \quad \forall v \in H^1(\mathcal{O}).\\
		&  
	\end{split}
\end{equation}
Now, by Lemmas \ref{lemmaa} and \ref{lémmaaprioriés} and taking subsequences if necessary, we obtain the following types of convergences:
	\begin{equation}\label{tyesofconv}
\begin{cases}
		\begin{split}
			&\tilde{u}^n \rightarrow \tilde{u}, &&\text{weakly* in } \mathcal{L}^2(\tilde{\Omega};L^\infty(Q_T)),\\
			&A(\tilde{u}^n) \rightarrow \overline{A}, &&\text{weakly in } \mathcal{L}^2(\tilde{\Omega};L^2(0,T;H^1(\mathcal{O}))),\\
			&\varepsilon_n \tilde{u}_{n} \rightarrow 0, &&\text{weakly in } \mathcal{L}^2(\tilde{\Omega};L^2(0,T;H^1(\mathcal{O}))),\\
			&\tilde{\mathcal{W}}^n \rightarrow \tilde{\mathcal{W}}  &&\text{strongly in } \mathcal{L}^2(\tilde{\Omega};C([0,T];\mathbb{H}_0),\\
			&\tilde{u}_0^n \rightarrow \tilde{u}_0  &&\text{strongly in } \mathcal{L}^2(\tilde{\Omega};L^2(\mathcal{O})).
		\end{split} 
\end{cases}
	\end{equation}
In addition, by recalling Embedding \eqref{émbéd}, we obtain that
$$
A(\tilde{u}^n) \rightarrow \overline{A}, \quad \text{strongly in }  \mathcal{L}^2(\Omega;L^2(0,T;L^2(\mathcal{O}))).
$$
{\color{black}Consequently, by taking advantage of the monotonicity of $A$ along with the fact that $A^{-1}$ is well defined and continuous, we derive from Lemmas \ref{lémmaaprioriés} and \ref{lémmatranslaté} and Aubin-Lions compactness lemma \cite{simon1986compact}}
$$
A(\tilde{u}^n) \rightarrow {A}(\tilde{u}), \quad \text{strongly in }  \mathcal{L}^2(\Omega;L^2(0,T;L^2(\mathcal{O}))),
$$
and consequently
$$
\tilde{u}^n \rightarrow \tilde{u}, \quad \text{strongly in }  \mathcal{L}^2(\Omega;L^2(0,T;L^2(\mathcal{O}))),
$$
which together with Convergence $\eqref{tyesofconv}_1$ yield that 
\begin{equation}\label{convlq}
\tilde{u}^n \rightarrow \tilde{u}, \quad \text{strongly in }  \mathcal{L}^2(\Omega;L^q(0,T;L^q(\mathcal{O}))), \; \forall q \in [1,\infty).
\end{equation}
Let $\mathcal{Z} \subset \Omega \times (0,T)$ be a measurable subset. Then, by multiplying  Identity \eqref{suitéidntild} by the characteristic function $\mathbbm{1}_{\mathcal{Z}}$, taking the expectation with respect to $\mathbb{\tilde{P}}$, and passing the limit as $n \rightarrow +\infty$ while keeping in mind $\eqref{tyesofconv}_6$, we obtain as $n \rightarrow +\infty$:
$$
\tilde{\mathbb{E}} \int_0^T \int_\mathcal{O} \tilde{u}_0^n v \mathbbm{1}_{\mathcal{Z}}(\omega,t)\; dxdt \rightarrow  \tilde{\mathbb{E}}\int_0^T \int_\mathcal{O} \tilde{u}_0 v \mathbbm{1}_{\mathcal{Z}}(\omega,t)\; dxdt, 
$$
where the law of $\tilde{u}_0$ is $\mu$ and $0 \leq \tilde{u}_0 \leq \overline{u}\; \mathbb{P}\text{-a.s}$.

\noindent
Moreover, knowing that $\nabla \Pi_n v \rightarrow \nabla v$  in $L^2(\mathcal{O})^d$ as $n \rightarrow \infty$, and taking $\eqref{tyesofconv}_2$ and \eqref{convlq} into account, we deduce that as $n \rightarrow +\infty$:
$$
\tilde{\mathbb{E}}\int_0^T \mathbbm{1}_{\mathcal{Z}}(\omega,t) \left(\int_{0}^t \int_{\mathcal{O}} a_{\varepsilon_n}(\tilde{u}^n(t))\nabla \tilde{u}^n(t) \cdot \nabla \Pi_n v\; dxds\right) dt \rightarrow \tilde{\mathbb{E}} \int_0^T \mathbbm{1}_{\mathcal{Z}}(\omega,t) \left(\int_{0}^t \int_{\mathcal{O}} a(\tilde{u}(t))\nabla \tilde{u}(t) \cdot \nabla  v\; dxds\right)dt.
$$
By the almost-sure boundedness of $\tilde{u}^n\left(\nabla \mathcal{K} \ast \tilde{u}^n\right)$ and $\eqref{tyesofconv}$, we also deduce that as $n \rightarrow +\infty$:
\begin{align*}
&\tilde{\mathbb{E}}  \int_0^T \mathbbm{1}_{\mathcal{Z}}(\omega,t) \left(\int_{0}^t \int_{\mathcal{O}} \tilde{u}^n(t)\left(\nabla \mathcal{K} \ast \tilde{u}^n(t)\right) \cdot \nabla \Pi_n v  dxds\right)dt\\ 
&\rightarrow  \tilde{\mathbb{E}}  \int_0^T \mathbbm{1}_{\mathcal{Z}}(\omega,t)\left(\int_{0}^t \int_{\mathcal{O}} \tilde{u}(t) \left(\nabla \mathcal{K} \ast \tilde{u}(t) \right) \cdot \nabla v dx ds\right)dt.
\end{align*}
Furthermore, we use the continuity of $f_M$ and $\eqref{convlq}$, along with the fact that $\Pi_n v \rightarrow v$  in $L^2(\mathcal{O})$ as $n \rightarrow \infty$. Additionally, we choose $M>0$ large enough to obtain that $f_M=f$ and as $n \rightarrow +\infty$:
\begin{align*}
\tilde{\mathbb{E}} \int_0^T \mathbbm{1}_{\mathcal{Z}}(\omega,t) \left( \int_{0}^t \int_{\mathcal{O}} f(\tilde{u}^n(t)) \Pi_n v\; dxds\right)dt&=\tilde{\mathbb{E}}  \int_0^T \mathbbm{1}_{\mathcal{Z}}(\omega,t) \left(\int_{0}^t \int_{\mathcal{O}} f_M(\tilde{u}^n(t)) \Pi_n v\; dxds\right)dt \\
&\rightarrow  \tilde{\mathbb{E}}  \int_0^T \mathbbm{1}_{\mathcal{Z}}(\omega,t) \left( \int_{0}^t \int_{\mathcal{O}} f_M(\tilde{u}(t)) v\; dxds \right)dt\\
&=\tilde{\mathbb{E}}  \int_0^T \mathbbm{1}_{\mathcal{Z}}(\omega,t) \left( \int_{0}^t \int_{\mathcal{O}} f(\tilde{u}(t)) v\; dxds\right)dt.
\end{align*}
We now move on to analyze the convergence of the stochastic term separately. On the one hand,  
\begin{align*}
	\int_{0}^T \Vert \sigma^n(\tilde{u}^n(t))-\sigma(\tilde{u}(t)) \Vert_{{L}_2(\mathbb{H} ;L^2(\mathcal{O}))}^2 dt &\leq  2	\int_{0}^T \Vert \sigma(\tilde{u}^n(t))-\sigma(\tilde{u}(t)) \Vert_{{L}_2(\mathbb{H} ;L^2(\mathcal{O}))}^2dt\\
	&+2\int_{0}^T \Vert \sigma^n(\tilde{u}^n(t))-\sigma(\tilde{u}^n(t)) \Vert_{{L}_2(\mathbb{H} ;L^2(\mathcal{O}))}^2dt.
\end{align*}
Clearly, the first term on the right-hand side converges to zero as $n \rightarrow +\infty$, as a  direct consequence of Assumption \ref{a4} and Convergence \eqref{convlq}. Let us now move on to the second term. To this end, we will use Lebesgue dominated convergence theorem. First, we make use of Assumption \ref{a5} and compute to obtain
\begin{align}\label{11}
		2\int_{0}^T \Vert \sigma^n(\tilde{u}^n(t))-\sigma(\tilde{u}^n(t)) \Vert_{{L}_2(\mathbb{H} ;L^2(\mathcal{O}))}^2dt&:=2 \sum_{k \in \mathbb{N}^*} \int_0^T \Vert\sigma_{k}^n (\tilde{u}^n(t))-\sigma_k(\tilde{u}^n(t))\Vert_{L^2(\mathcal{O})}^2 dt \nonumber\\
		&=2 \sum_{k \in \mathbb{N}^*} \int_0^T \Vert \Pi_n(\sigma_{k} (\tilde{u}^n(t)))-\sigma_k(\tilde{u}^n(t))\Vert_{L^2(\mathcal{O})}^2 dt\nonumber\\
		&\leq 4 \sum_{k \in \mathbb{N}^*} \int_0^T \Vert \sigma_k(\tilde{u}^n(t))\Vert_{L^2(\mathcal{O})}^2 dt \nonumber \\
		&\leq 4 \int_0^T \Vert \sigma(\tilde{u}^n(t))\Vert_{L_2(\mathbb{H};L^2(\mathcal{O}))}^2 dt\nonumber\nonumber \\
		&\leq 4C_1 \int_0^T \left(1+\Vert \tilde{u}^n(t)\Vert_{L^2(\mathcal{O})}^2 \right) dt.
\end{align}
Additionally, from Lemma \ref{lemmaa}, the last term in Inequality \eqref{11} can be estimated independently of $n$. Hence, by taking into account the fact that as $n \rightarrow +\infty$:
$$
\Pi_n\left(\sum_{k \in \mathbb{N}^*} \sigma_{k}(\tilde{u}^n(t))\right) \rightarrow \sum_{k \in \mathbb{N}^*} \sigma_{k}(\tilde{u}(t)) \quad \text{ strongly in } L^2(\mathcal{O}),
$$
the dominated convergence theorem allows us to deduce that as $n \rightarrow +\infty$:
$$
2\int_{0}^T \Vert \sigma^n(\tilde{u}^n(t))-\sigma(\tilde{u}^n(t)) \Vert_{{L}_2(\mathbb{H} ;L^2(\mathcal{O}))}^2dt \rightarrow 0 \quad \mathbb{P}\text{-a.s}.
$$
Thus, as $n \rightarrow +\infty$:
\begin{equation}\label{convsigma}
	\int_{0}^T \Vert \sigma^n(\tilde{u}^n(t))-\sigma(\tilde{u}(t)) \Vert_{{L}_2(\mathbb{H} ;L^2(\mathcal{O}))}^2 dt  \rightarrow 0 \quad \mathbb{P}\text{-a.s}.
\end{equation}
By combining Convergences $\eqref{tyesofconv}_5$ and \eqref{convsigma}, and passing to a subsequence if necessary, we use \cite[Lemma 2.1]{debussche2011local} to obtain as $n \rightarrow +\infty$:
\begin{equation}\label{partie1}
\int_0^t \sigma^n(\tilde{u}^n(s)) d\tilde{\mathcal{W}}^n_s \rightarrow  \int_0^t \sigma(\tilde{u}(s)) d\tilde{\mathcal{W}}_s\quad \mathbb{P}\text{-a.s}.
\end{equation}
Now, let $p \in (2,q]$. Then, by BDG inequality and taking Assumption \ref{a5} as well as Condition \eqref{condinit} into account, we obtain
\begin{align}\label{vitalicond}
	\tilde{\mathbb{E}}{\left\|\int_0^t \sigma^n\left(\tilde{u}^n(s)\right) d \tilde{\mathcal{W}}^{n}_s\right\|_{L^2\left(0,T; L^2(\mathcal{O})\right)}^p}& =\tilde{\mathbb{E}}\left(\int_0^T\left\|\sum_{k=1}^n \int_0^t \sigma_{k}^n\left(\tilde{u}^n(s)\right) d \tilde{\mathcal{W}}_{s}^k\right\|_{L^2(\mathcal{O})}^2 d t\right)^{\frac{p}{2}}\nonumber \\
	& \leq {C}_3 \tilde{\mathbb{E}}\left[\sup _{t \in[0, T]}\left\|\sum_{k=1}^n \int_0^t \sigma_{ k}^n\left(\tilde{u}^n(s)\right) d \tilde{\mathcal{W}}_{s}^k\right\|_{L^2(\mathcal{O})}^p\right]  \nonumber \\ 
	& \leq C_3 \tilde{\mathbb{E}}\left[\left(\int_0^T \sum_{k=1}^n\left\|\sigma_{k}^n\left(\tilde{u}^n(t)\right)\right\|_{L^2(\mathcal{O})}^2 d t\right)^{\frac{p}{2}}\right]\nonumber \\
	&<\infty. 
\end{align}
Combining  Convergence \eqref{convsigma}, Estimate \eqref{vitalicond} and Corollary \ref{cor:Lq0-est}, by Vitali's convergence theorem \cite[Theorem 24]{dhariwal2019global}, we deduce that as $n \rightarrow +\infty$: 
\begin{equation}\label{lastconv}
\int_0^t \sigma^n(\tilde{u}^n(s)) d\tilde{\mathcal{W}}^n_s \rightarrow  \int_0^t \sigma(\tilde{u}(s)) d\tilde{\mathcal{W}}_s\quad \text{ strongly in } \mathcal{L}^2(\tilde{\Omega};L^2(0,T;L^2(\mathcal{O}))).
\end{equation}
Consequently, by Convergence \eqref{lastconv} along with the fact that $\Pi_n v \rightarrow v$  in $L^2(\mathcal{O})$ as $n \rightarrow \infty$, we derive that as $n \rightarrow \infty$
$$
\mathbb{E} \int_{0}^T  \mathbbm{1}_{\mathcal{Z}}(\omega,t)  \left(\int_{0}^t  \int_{\mathcal{O}} \sigma^n(\tilde{u}^n(s)) \Pi_nv\; dx d\tilde{\mathcal{W}}^n_s\right)dt \rightarrow \mathbb{E}  \int_{0}^T  \mathbbm{1}_{\mathcal{Z}}(\omega,t) \left(\int_{0}^t \int_{\mathcal{O}} \sigma(\tilde{u}(s)) v\; dx d\tilde{\mathcal{W}}_s\right)dt.
$$
Combining all the previously-established convergences and letting $n \rightarrow \infty$ in Identity \eqref{suitéidntild}, we conclude that  
\begin{align*}
	\mathbb{E} \int_0^T \mathbbm{1}_{\mathcal{Z}}(\omega,t)&\left( \int_\mathcal{O} \tilde{u}(t)v\; dx-\int_\mathcal{O} \tilde{u}_0 v\; dx+ \int_{0}^t \int_{\mathcal{O}} a(\tilde{u}(s))\nabla \tilde{u}(s) \cdot \nabla v\; dxdt-\int_{0}^t \int_{\mathcal{O}} \tilde{u}(s)\left(\nabla \mathcal{K} \ast \tilde{u}(s)\right) \cdot \nabla v\; dxdt \right.\\
	&\left. -\int_{0}^t \int_{\mathcal{O}} f(\tilde{u}(s)) v\; dxdt-\int_{0}^t \int_{\mathcal{O}} \sigma(\tilde{u}(s)) v\; dx d\tilde{\mathcal{W}}_s\right)dt=0, \quad \forall v \in H^1(\mathcal{O}), \; \forall t \in [0,T].
\end{align*}
This concludes the proof.
\end{proof}
\begin{rem}
Proceeding by the exact same techniques used in the proof of Lemma \ref{lemmaa}, one can establish that $(\tilde{u}(t))_{0 \leq t \leq T}$ remains almost surely positive and bounded. 
\end{rem}
\subsection{Uniqueness}
Now that we have addressed the question of existence of weak martingale solutions. We proceed to discuss uniqueness issues. Namely, we address the following question:

\noindent 
\textit{Given two stochastic weak martingale solutions $\left(\tilde{u}_1,\tilde{\mathcal{S}}\right)$ and $\left(\tilde{u}_2,\tilde{\mathcal{S}}\right)$ with a same stochastic basis $\tilde{\mathcal{S}}$ which are obtained as limits of Skorokhod's representations. Do $\tilde{u}_1$ and $\tilde{u}_2$ coincide in a pathwise sense?}

\noindent
The answer to the above question is provided in the following lemma:

\begin{lem}
	Let Assumptions \ref{a1}-\ref{a5} be satisfied and consider two stochastic weak martingale solutions $\left(\tilde{u}_1,\tilde{\mathcal{S}}\right)$ and $\left(\tilde{u}_2,\tilde{\mathcal{S}}\right)$ with the same stochastic basis $\tilde{\mathcal{S}}$, and starting from the same initial condition $\tilde{u}_0 \in \mathcal{L}^2(\tilde{\Omega};L^2(\mathcal{O}))$ given by Convergence $\eqref{tyesofconv}_6$. Then, it holds that
	$$
	\tilde{u}_1=\tilde{u}_2 \; \tilde{\mathbb{P}}\text{-a.s.} \; \text{in } {Q}_T.
	$$
\end{lem}
\begin{proof}
Consider the stochastic process $U:=\tilde{u}_1-\tilde{u}_2$. Then, clearly, $\forall t \in [0,T]:$
\begin{equation}\label{diff}
	\begin{split}
		\int_\mathcal{O} U(t)v\; dx&=- \int_{0}^t \int_{\mathcal{O}} \left(\nabla A(\tilde{u}_1)-\nabla A(\tilde{u}_2)\right) \cdot \nabla v\; dxdt\\
		&+\int_{0}^t \int_{\mathcal{O}} \left(\tilde{u}_1(t)\left(\nabla \mathcal{K} \ast \tilde{u}_1(t)\right)-\tilde{u}_2(t)\left(\nabla \mathcal{K} \ast \tilde{u}_2(t)\right)\right) \cdot \nabla v\; dxdt\\
		&+\int_{0}^t \int_{\mathcal{O}} \left(f(\tilde{u}_1(t))-f(\tilde{u}_1(t))\right)v\; dxdt\\
		&+\int_{0}^t \int_{\mathcal{O}} \left(\sigma(\tilde{u}_1(t))-\sigma(\tilde{u}_2(t))\right)v\; dx d\tilde{\mathcal{W}}_s, \quad \forall v \in H^1(\mathcal{O}). \; 
	\end{split}
\end{equation}
The key idea is to take computational advantages of the duality between $U$ and an adequately-chosen function $w \in H^2(\mathcal{O})$. To this end, let $(\omega,t) \in \tilde{\Omega}  \times (0,T)$ and consider the following elliptic problem:
\begin{equation}\label{élliptic}
	\begin{cases}
	\begin{split}
		&-\Delta w=U(\omega,t,.), \quad &\text{ in } \mathcal{O},\\
		&\nabla w \cdot \overrightarrow{\textbf{n}}=0, &\text{ on } \partial \mathcal{O}. 
	\end{split}
	\end{cases}
\end{equation}
Since $U(\omega,t.) \in L^\infty(\mathcal{O})$. Then, by the standard elliptic theory, Problem \eqref{élliptic} has a unique solution $w$ satisfying the following regularity:
$$
w(\omega,.,.) \in C([0,T];H^2(\mathcal{O})) \text{ and } \int_{\mathcal{O}} w(\omega,t,.) dx=0.$$
On the one hand, the use of Problem \eqref{élliptic} and Green's formula lead to the identity
\begin{equation}\label{dualidén}
\int_\mathcal{O} U(t)w(\omega,t,.)\; dx=\int_{\mathcal{O}} \vert \nabla w(\omega,.,t) \vert^2 dx.
\end{equation}
On the other hand, from Identity \eqref{diff} with $v=w(\omega,t,.)$ and Problem \eqref{élliptic}, we use Green's formula, Cauchy-Schwartz and Young's inequalities while keeping in mind the almost-sure essential boundedness of $\tilde{u}_1$ and $\tilde{u}_2$, Assumptions \ref{a1}--\ref{a3} and the local Lipschitz property of $f$ and $A$ to obtain 
\begin{align*}
		\tilde{\mathbb{E}} \int_\mathcal{O} U(t)w(\omega,t,.)\; dx &\leq -C_{31} \tilde{\mathbb{E}} \int_0^t \int_{\mathcal{O}} \vert U(t)\vert^2 dx dt+2\tilde{\zeta} \tilde{\mathbb{E}} \int_0^t \int_\mathcal{O}  \vert U(t)\vert^2 dx dt+\dfrac{1}{4 \tilde{\zeta}} \tilde{\mathbb{E}} \int_0^t \int_\mathcal{O} \vert \nabla w(\omega,t,.) \vert^2 dxdt\\
		&+\dfrac{1}{4 \tilde{\zeta}} \tilde{\mathbb{E}}\int_0^t \int_{\mathcal{O}} \vert w(\omega,t,.) \vert^2 dxdt, \quad \forall \tilde{\zeta}>0.
\end{align*}

\noindent
Hence, by choosing $0<\tilde{\zeta}<\dfrac{C_{31}}{2}$ and by applying Poincaré-Wirtinger inequality to $w(\omega,t,.)$, we derive
\begin{align*}
\tilde{\mathbb{E}}	\int_{\mathcal{O}} U(t) w(\omega,t,.) dx 
&\leq C_{32} \tilde{\mathbb{E}} \int_0^t \int_{\mathcal{O}} \vert \nabla w(\omega,t,.) \vert^2 dx.
\end{align*}

\noindent
Hence, from Identity \eqref{dualidén}, we acquire the following inequality:
\begin{align*}
	\tilde{\mathbb{E}}	\int_{\mathcal{O}} \vert \nabla w(\omega,.,t) \vert^2 dx 
	&\leq C_{31} \tilde{\mathbb{E}} \int_0^t \int_{\mathcal{O}} \vert \nabla w(\omega,t,.) \vert^2 dx.
\end{align*}
Thereby, by recalling that $\nabla w (\omega,.,0)=U(\omega,.,0)=0 \text{ in } \mathcal{O}$, a direct application of Gronwall's inequality leads to
	$$
\nabla w=0 \; \tilde{\mathbb{P}}\text{-a.s.} \; \text{ in } Q_T.
$$
Consequently, multiplying  $\eqref{élliptic}_1$ by $U,$ integrating over $\mathcal{O}$ and using Green's formula yield the desired result.
\end{proof}
\section{A numerical illustration}
\label{section8}
The aim of this section is to give graphical illustrations of Model \eqref{basicstoch} in order to show effect of stochastic Gaussian noise on biological aggregation. To this end, we consider the fixed spatial domain
\[
\mathcal{O}=(-4,4)^2
\]
and the fixed time horizon
\[
(0,T)=(0,12)
\]
and starts from an initial state given by 
\begin{align*}
	u_0(x,y)&=2\exp(-((x+1)^2 +x^2))+1.5\exp(-(x^2+(y-1)^2))\\
	&+2\exp(-((x-1.5)^2+(y+ 1)^2)), \quad \forall (x,y) \in (-4,4)^2.
\end{align*}
Additionally, we consider the following two-sidedly degenerate density-dependent diffusion rate:
$$
a(u):=u(\overline{u}-u), \quad \forall u \in \mathbb{R}^+,
$$
where $\overline{u}>0$ is large enough such that 
$$\Vert u_0 \Vert_{L^\infty((-4,4)^2)} < \overline{u}.$$
As for the aggregation kernel, we consider
$$
\mathcal{K}(x,y)=\dfrac{1}{\displaystyle \int_{\mathbb{R}} \int_{\mathbb{R}} \exp(-\dfrac{x^2+y^2}{2}) dx dy}\exp(-\dfrac{x^2+y^2}{2}), \quad \forall (x,y)\in \mathbb{R}^2. 
$$
\noindent
Now, we introduce a uniform spatial mesh on \(\mathcal{O}\):
\[
x_i = -4 + i\Delta x, \quad i=0,\dots,N_x,
\qquad
y_j = -4 + j\Delta y, \quad j=0,\dots,N_y,
\]
where
\[
\Delta x = \frac{8}{N_x}, \qquad \Delta y = \frac{8}{N_y}.
\]
The temporal grid is
\[
t^n = n\Delta t, \quad n=0,\dots,N_T, 
\qquad \Delta t = \frac{12}{N_T}.
\]

\noindent
The fully discrete scheme (finite differences in space, extended Milstein in time \cite{hu2022stability}) is:

\begin{equation}\label{discretesystem}
	\begin{cases}
		\begin{aligned}
			u_{i,j}^{\,n+1} &= u_{i,j}^{\,n}
			+ \Delta t\Big[(L_h(u^{n}))_{i,j} + f\!\big(u_{i,j}^{\,n}\big)\Big]
			+ \sigma\!\big(u_{i,j}^{\,n}\big)\,\Delta W_{i,j}^{\,n} \\[1ex]
			&\quad + \tfrac12\,\sigma\!\big(u_{i,j}^{\,n}\big)\,\sigma'\!\big(u_{i,j}^{\,n}\big)
			\left[\big(\Delta W_{i,j}^{\,n}\big)^2 - \Delta t\right],
		\end{aligned} \\[4ex]
		
		\begin{aligned}
			(L_h(u))_{i,j}
			&= \frac{F^{x}_{i+\tfrac12,j}-F^{x}_{i-\tfrac12,j}}{\Delta x}
			+ \frac{F^{y}_{i,\,j+\tfrac12}-F^{y}_{i,\,j-\tfrac12}}{\Delta y},
		\end{aligned} \\[4ex]
		
		\begin{aligned}
			F^{x}_{i+\tfrac12,j} &= a_{i+\tfrac12,j}\,\frac{u_{i+1,j}-u_{i,j}}{\Delta x}
			- u_{i+\tfrac12,j}\,\frac{v_{i+1,j}-v_{i,j}}{\Delta x}, \\[1ex]
			F^{y}_{i,\,j+\tfrac12} &= a_{i,\,j+\tfrac12}\,\frac{u_{i,j+1}-u_{i,j}}{\Delta y}
			- u_{i,\,j+\tfrac12}\,\frac{v_{i,j+1}-v_{i,j}}{\Delta y},
		\end{aligned} \\[4ex]
		
		\begin{aligned}
			u_{i+\tfrac12,j} &= \tfrac12\big(u_{i+1,j}+u_{i,j}\big), \qquad
			u_{i,\,j+\tfrac12} = \tfrac12\big(u_{i,j+1}+u_{i,j}\big), \\[1ex]
			a_{i+\tfrac12,j} &= \tfrac12\big(a(u_{i+1,j})+a(u_{i,j})\big), \qquad
			a_{i,\,j+\tfrac12} = \tfrac12\big(a(u_{i,j+1})+a(u_{i,j})\big),
		\end{aligned} \\[4ex]
		
		\begin{aligned}
			v_{i,j}^{\,n} &= \sum_{p=0}^{N_x}\sum_{q=0}^{N_y}
			\mathcal{K}_{\,i-p,\,j-q}\,u_{p,q}^{\,n}\,\Delta x\,\Delta y,
		\end{aligned} \\[4ex]
		
		\begin{aligned}
			F^{x}_{-\tfrac12,j} &= F^{x}_{N_x+\tfrac12,j}=0, \qquad
			F^{y}_{i,-\tfrac12}=F^{y}_{i,N_y+\tfrac12}=0,
		\end{aligned} \\[4ex]
		
		\begin{aligned}
			\text{and the discrete initial condition:}\qquad
			u_{i,j}^0 &= u_0(x_i,y_j), \quad i=0,\dots,N_x,\; j=0,\dots,N_y,
		\end{aligned}
	\end{cases}
\end{equation}
with the following discrete notations:
\begin{itemize}
	\item[$\bullet$] \(u_{i,j}^n \approx u(x_i,y_j,t^n)\) denotes the numerical approximation at node \((x_i,y_j)\) and time \(t^n\).
	\item[$\bullet$] \(u_{i,j}^0 = u_0(x_i,y_j)\) is the discrete initial condition.
	\item[$\bullet$] \(\Delta W_{i,j}^n\) are independent Gaussian increments with
	\(\mathbb{E}[\Delta W_{i,j}^n]=0,\; \mathbb{E}[(\Delta W_{i,j}^n)^2]=\Delta t.\)
	\item[$\bullet$] \(v_{i,j}^n\) denotes the discrete convolution \((\mathcal{K}*u^n)(x_i,y_j)\).
\end{itemize}

\noindent

\noindent
 Based on the obtained numerical results, we can assert the following interpretations:
\subsection{Aggregation without environmental noise}
Figure \ref{fig1} illustrates the dynamics of aggregation in the absence of external perturbations. Starting from two distinct high-density clusters, a merging process occurs due to forces of attraction and diffusion. At $t=4$, the clusters begin to attract each other, and by $t=8$, they merge into a single larger aggregate. By $t=12$, the system stabilizes into a diffused yet cohesive configuration. This progression is a hallmark of natural aggregation phenomena, such as cellular chemotaxis or protein clustering, driven purely by intrinsic attraction forces and unaffected by external constraints.
\subsection{Aggregation with proportional-shifted environmental noise}
Figure \ref{fig2} presents aggregation dynamics in the presence of proportional-shifted environmental noise, showcasing a contrast to the smooth merging process seen in Fig. \ref{fig1}. Although both begin with the same initial clusters, the bounded noise causes a rapid transition to a concentrated, uniform high-density region with sharp edges. This evolution, apparent from $t=0$ to $t=12$, reveals how population constraints and external perturbations reshape the dynamics, suppressing diffusion and maintaining distinct boundaries. Such patterns resemble biological processes like bacterial biofilm formation or cellular aggregation in constrained environments, where density and resource limitations create unique aggregation behaviors. One can also observe from Fig. \ref{fig6} the effect of proportional-shifted environmental noise on the total mass  density.
\subsection{Aggregation with periodic environmental noise}
Figure \ref{fig3} explores aggregation in the presence of periodic environmental noise, resulting in a fundamentally different evolution compared to the smooth dynamics of Fig. \ref{fig1}. Starting from identical initial clusters, periodic perturbations introduce a dynamic reorganization that disrupts clean merging. From $t=0$ to $t=12$, the aggregate maintains overall cohesion while continuously adapting its internal structure. This behavior contrasts with the uniform merging of the noise-free case and mirrors biological systems like cell clusters under mechanical stress or bacterial colonies in oscillating environments, where structures must balance collective integrity with adaptability to external fluctuations. One can also observe from Fig. \ref{fig6} the effect of periodic environmental noise, in comparison to proportional-shifted environmental noise, on the total mass  density.
\subsection{Effect of a stochastically-perturbed initial condition on aggregation}
The comparison of Figs. \ref{fig2}-\ref{fig5} reveals key insights into how our stochastic aggregation-diffusion system behaves under different conditions. Indeed, Fig. \ref{fig2} demonstrates classical aggregation behavior, characterized by well-defined boundaries and uniform interior density. When stochastic initial conditions are introduced (Fig. \ref{fig4}), the system maintains its aggregative nature but develops internal heterogeneity while preserving overall connectivity. The transition to sine-based noise yields notably different patterns. Figure \ref{fig3}  (deterministic initial conditions) shows partial fragmentation, while Fig. \ref{fig5} (stochastic initial conditions) exhibits enhanced cluster separation and resistance to complete aggregation. On the other hand, Fig. \ref{fig7} provides an illustration of the effect of a stochastically-perturbed initial condition on the total mass density.
\begin{figure}[H]
	\hspace*{-2.6cm}
	\includegraphics[width=22cm,height=8.5cm]{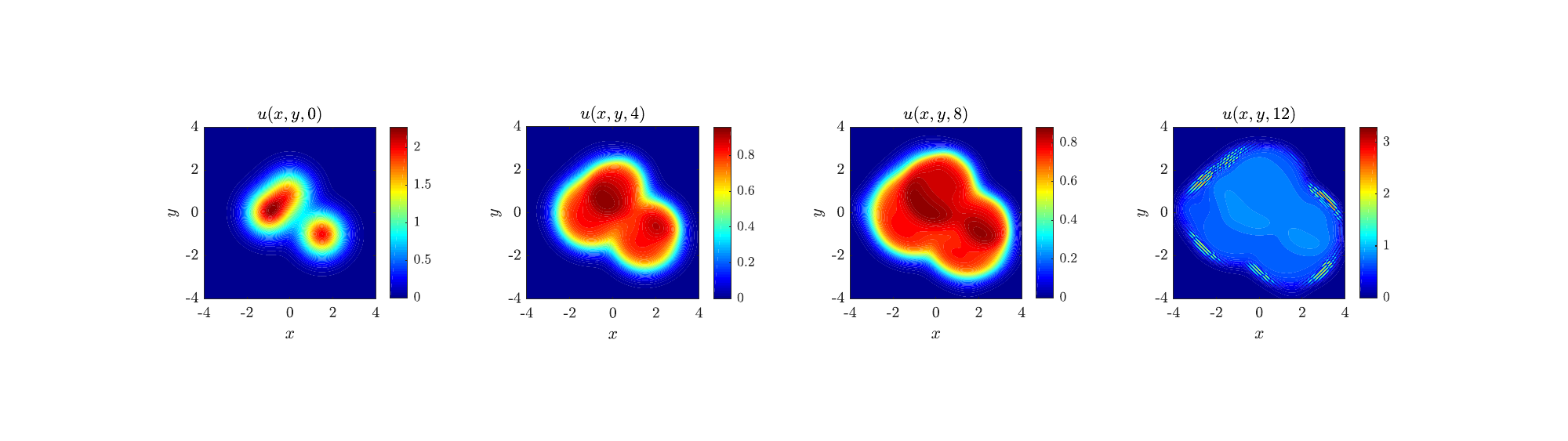}
	\caption{{Evolution of the density $u$ at four points in time: $t=0$, $t=4$, $t=8$ and $t=12$ with $\alpha=0.4$ and $\mu=0.5$ in the absence of stochastic noise $(\sigma \equiv 0)$}.}
	\label{fig1}
\end{figure}
\vspace*{-0.6cm}
\begin{figure}[H]
	\hspace*{-2.6cm}
	\includegraphics[width=22cm,height=8.5cm]{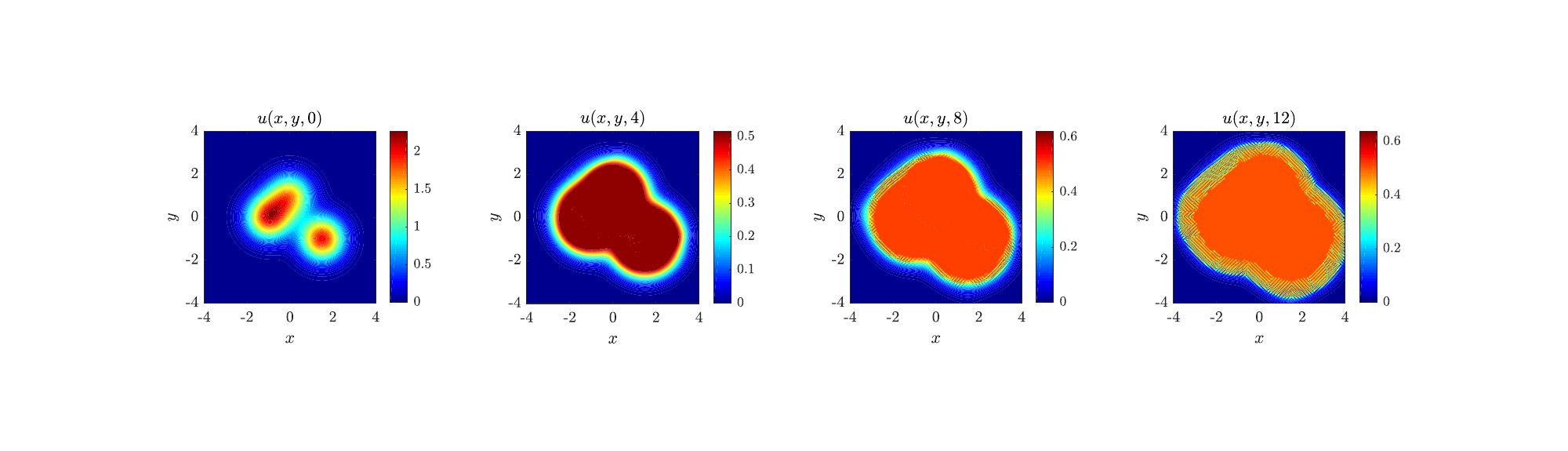}
	\caption{{Evolution of the density $u$ at four points in time: $t=0$, $t=4$, $t=8$ and $t=12$, with $\alpha=0.4$ and $\mu=0.5$ in the presence of the stochastic noise $\sigma(u)=1.2 \min\{u,\overline{u}-u\}$}.}
		\label{fig2}
\end{figure}
\vspace*{-0.6cm}
\begin{figure}[H]
	\hspace*{-2.6cm}
	\includegraphics[width=22cm,height=9cm]{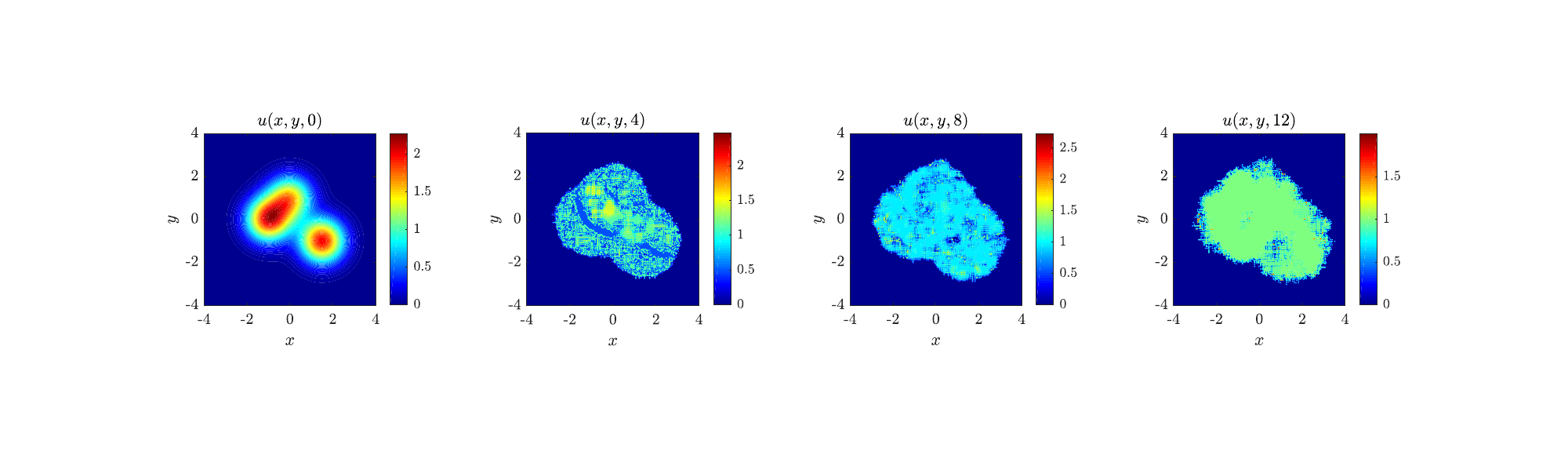}
	\caption{{Evolution of the density $u$ at four points in time: $t=0$, $t=4$, $t=8$ and $t=12$, with $\alpha=0.4$ and $\mu=0.5$ in the presence of the stochastic noise $\sigma(u)=1.2 \sin(\dfrac{\pi u}{\overline{u}})$}.}
		\label{fig3}
\end{figure}
\begin{figure}[H]
	\hspace*{-2.6cm}
	\includegraphics[width=22cm,height=9cm]{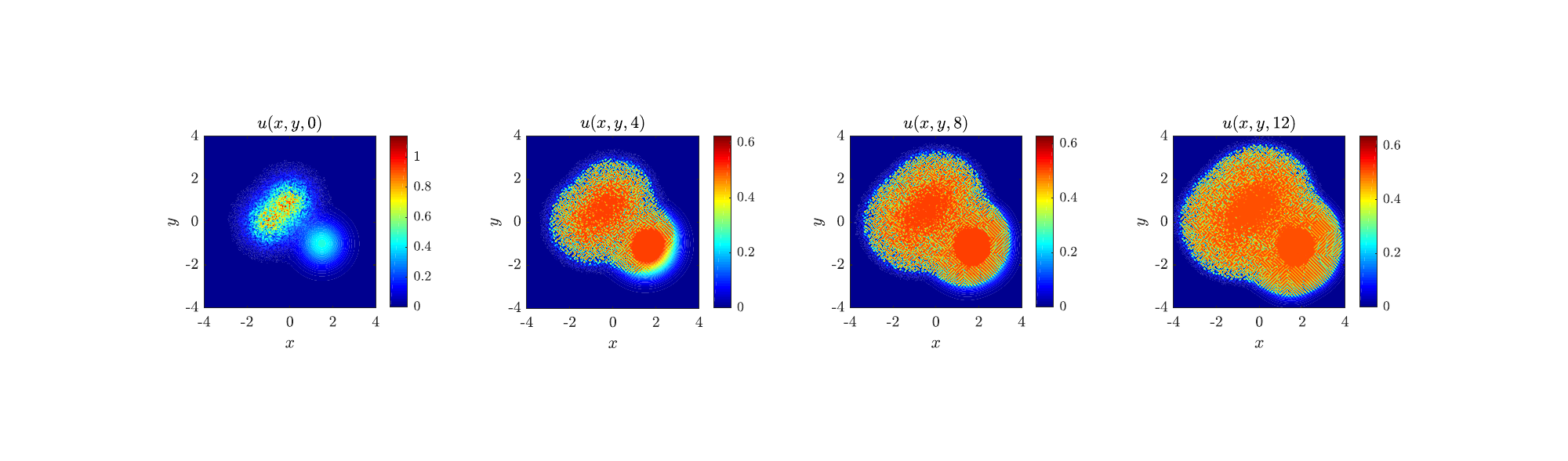}
	\caption{{Evolution of the density $u$ at four points in time: $t=0$, $t=4$, $t=8$ and $t=12$, with $\alpha=0.4$ and $\mu=0.5$ in the presence of the stochastic noise $\sigma(u)=1.2 \min\{u,\overline{u}-u\}$} for a stochastically-perturbed initial condition $u_0$.}
	\label{fig4}
\end{figure}
\begin{figure}[H]
	\hspace*{-2.6cm}
	\includegraphics[width=22cm,height=9cm]{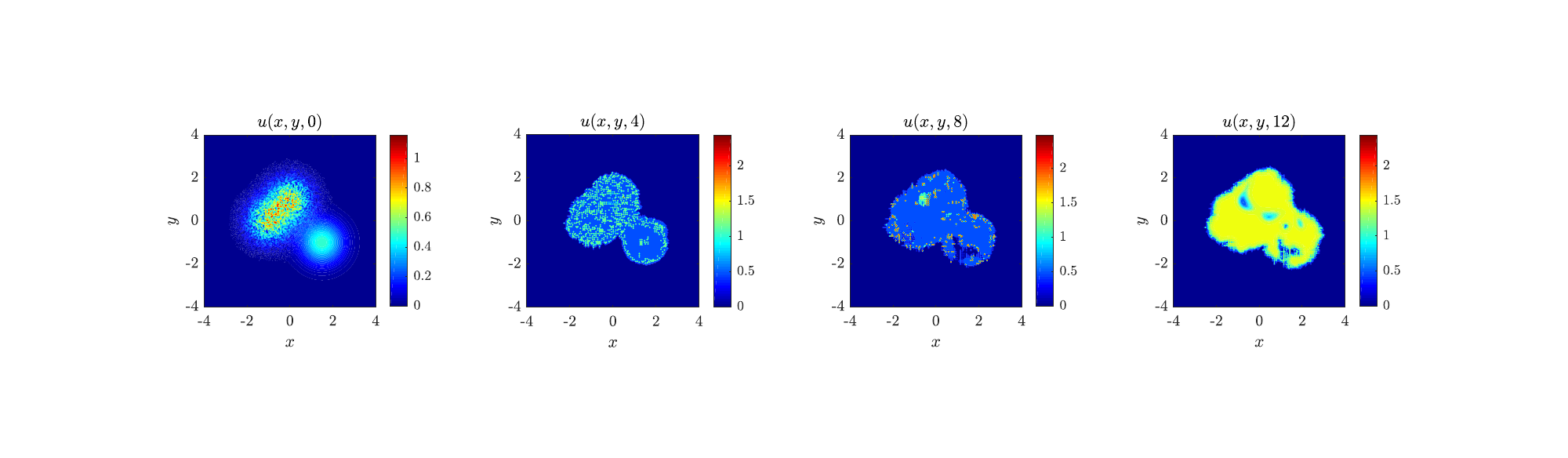}
	\caption{{Evolution of the density $u$ at four points in time: $t=0$, $t=4$, $t=8$ and $t=12$, with $\alpha=0.4$ and $\mu=0.5$ in the presence of the stochastic noise $\sigma(u)=1.2 \sin(\dfrac{\pi u}{\overline{u}})$ for a stochastically-perturbed initial condition $u_0$}.}
	\label{fig5}
\end{figure}
\begin{figure}[H]
	\centering
	\includegraphics[width=12cm,height=10cm]{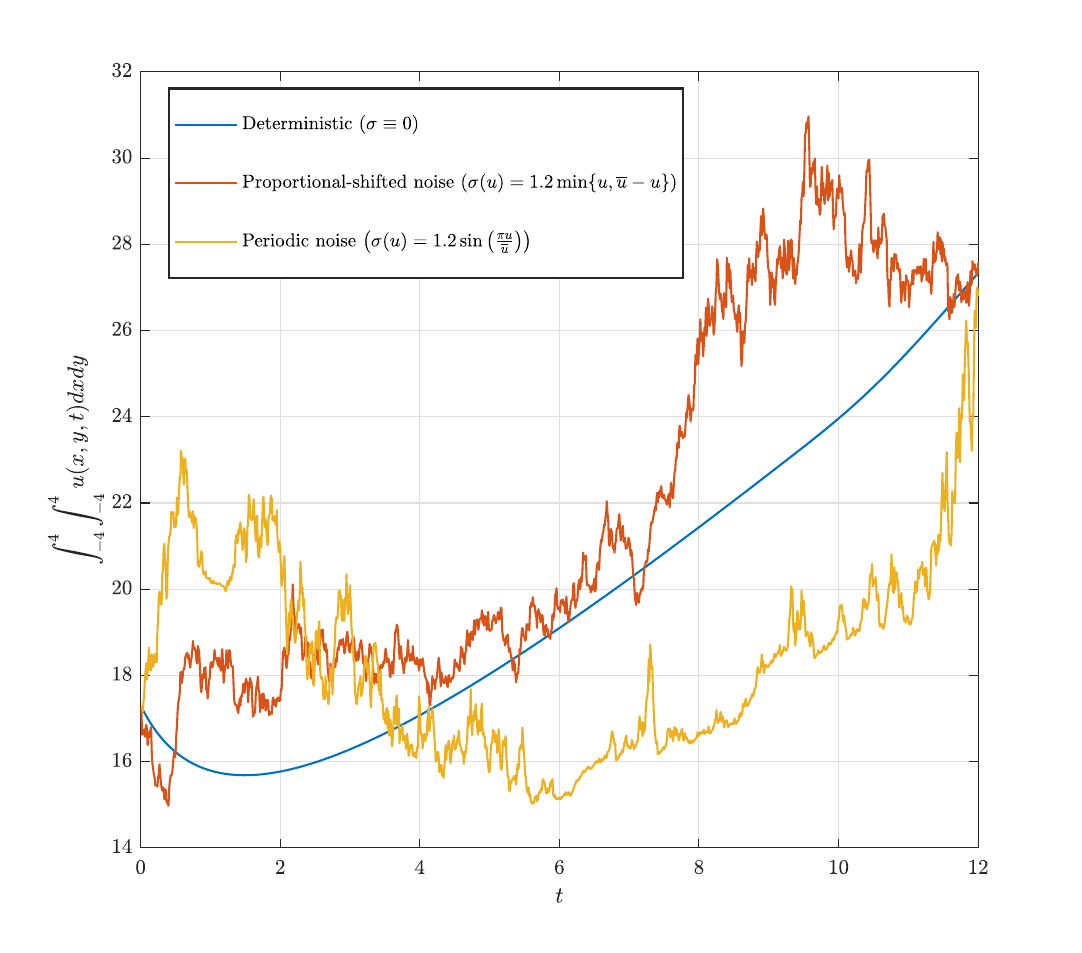}
	\caption{{Evolution of the total mass of $u$ with respect to time in the deterministic and stochastic cases with $\alpha=0.4$, $\mu=0.5$  for a deterministic initial condition $u_0$}.}
	\label{fig6}
\end{figure}
\begin{figure}[H]
	\centering
	\includegraphics[width=12cm,height=10cm]{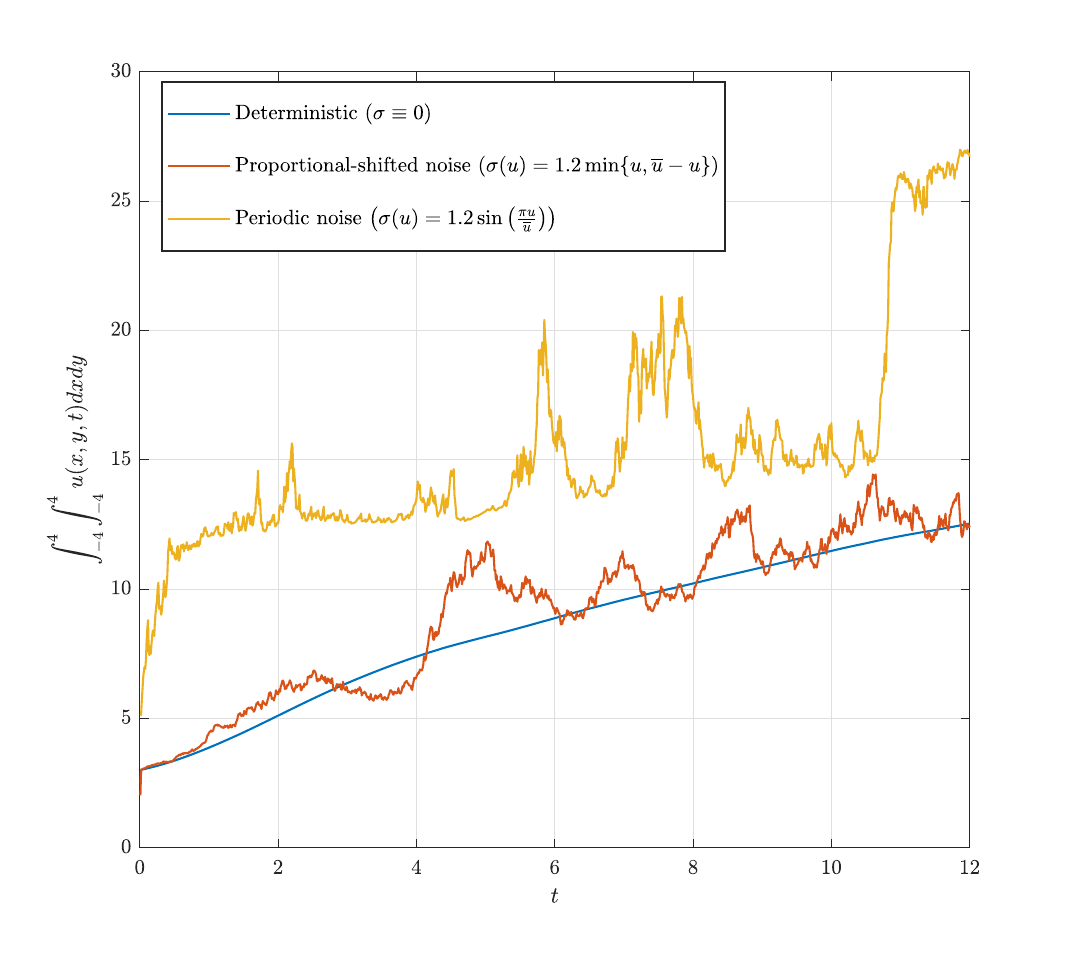}
	\caption{{Evolution of the total mass of $u$ with respect to time in the deterministic and stochastic cases with $\alpha=0.4$ and $\mu=0.5$ for a stochastically-perturbed initial condition $u_0$}.}
	\label{fig7}
\end{figure}
\section{Conclusion and future work}
\label{section9}
In this paper, we have presented new results on the well-posedness of aggregation-diffusion equations, particularly in the less-explored stochastic setting. The considered equation featured locally Lipschitz reaction terms incorporating growth and competition, a two-sidedly degenerate density-dependent diffusion rate accounting for the volume-filling effect, and a stochastic term representing environmental perturbations. Our approach differed from earlier contributions in the deterministic \cite{bendahmane2023optimal} and stochastic \cite{tang2024strong} cases, by introducing a novel methodology to establish well-posedness. Specifically, we first derived uniform a priori estimates and ensured almost-sure positiveness and boundedness for a perturbed non-degenerate version of the model. Subsequently, Prokhorov's compactness theorem and Skorokhod's representation theorem were employed to pass to the weak limit in the perturbed system, yielding the desired solution. Additionally, we addressed the uniqueness of the solution in a particular case based on a duality technique.

The numerical results highlighted that Gaussian noise introduces additional spatial complexity without disrupting  the system's fundamental basic aggregative tendencies.  However, the influence of Lévy noise, which models large-scale environmental perturbations \cite{mehdaoui2024well}, remains unclear and could potentially introduce significant variability or alter convergence. Meanwhile, the observed convergence toward steady-state configurations (Figs.\;\ref{fig3}-\ref{fig5}) underscored the importance of stability and asymptotic behavior in the considered equation. Thus, understanding the interplay between complex noise structures, such as Lévy noise, and aggregation mechanisms, as well as rigorously analyzing the long-term behavior and stability of solutions, remains worthy of further investigation. These topics will be subjects of future work.
\section*{Statements and Declarations}
\subsection*{Acknowledgments}
M. Mehdaoui gratefully acknowledges the support of the French Embassy in Morocco and the Moroccan Center for Scientific and Technical Research (CNRST) for funding his three-month research visit to the Institut de Mathématiques de Bordeaux (IMB), Université de Bordeaux, France. He also expresses his sincere gratitude to the IMB for their warm hospitality and to the Euromed University of Fez (UEMF) for their valuable logistical support.
\subsection*{Data availability}
Not applicable.
\subsection*{Funding}
Not applicable.
\subsection*{Conflict of interest}
The authors declare that they have no conflicts of interest.

\bibliographystyle{abbrvurl}
\bibliography{BMT_WEAK_MARTINGALE}
\end{document}